\documentclass[a4paper,11pt,fleqn, pdflatex]{article}
\usepackage{authblk}

\usepackage{amsmath, amsthm, amssymb}
\usepackage{csquotes} %for nice-looking quotations marks/environments.
\usepackage{caption} %for captions
\usepackage{txfonts,pxfonts} %for counterfactual arrows etc. 
\usepackage{color} %for colored texts
\usepackage[hidelinks]{hyperref}

\usepackage{tikz,pgf} 
\usepackage{tikz-cd}

\newtheorem*{theorem*}{Theorem}
\newtheorem{theorem}{Theorem}

\newtheorem{corollary}{Corollary}

\newtheorem{lemma}{Lemma}

\newtheorem*{lemma*}{Lemma}
\newtheorem*{remark*}{Remark}

\theoremstyle{definition}
\newtheorem{definition}{Definition}

\DeclareMathOperator{\Ext}{\mathsf{Ext}}
\DeclareMathOperator{\NExt}{\mathsf{NExt}}
\DeclareMathOperator{\RExt}{\mathsf{RExt}}
\DeclareMathOperator{\nExt}{\mathsf{nExt}}
\DeclareMathOperator{\rExt}{\mathsf{rExt}}

\usepackage{setspace}
\usepackage[left=3cm,right=3cm,top=3cm,bottom=3cm]{geometry}
\allowdisplaybreaks[1]
\usepackage[
  backend=biber,
  style=numeric-comp,
  sorting=nyt,
  giveninits=true,
  isbn=false,
  doi=false,
  date=year
]{biblatex}
\title{Post completeness in conditional logic} 

\author{Giuliano Rosella$^1$ and Yale Weiss$^2$}

\date{$^1$Institute of Computer Science \\ 
Czech Academy of Sciences \\ 
rosella@cs.cas.cz\\
\smallskip
$^2$The Saul Kripke Center and Philosophy Program \\
    The Graduate Center, CUNY \\ yweiss@gradcenter.cuny.edu
}

\begin{document}

\maketitle

  %%%% INSERT EACH AUTHOR'S FIRST INITIAL AND SURNAME%%%%
 % \runauthor{G.~Rosella and Y.~Weiss}

\begin{abstract} A logic is Post complete if it is consistent but has no consistent proper extensions. In this article, we systematically investigate Post completeness in conditional logic. We establish some results on the structure of various lattices of conditional logics and examine Post completeness from this perspective. We identify all of the finitely many regular and normal Post complete conditional logics, and prove analogues of Makinson’s embedding theorems for certain proper subfamilies of normal conditional logics. We also show that the analogue of Makinson's embedding theorem fails for normal conditional logic by producing an incomplete normal system. Furthermore, we show that certain basic conditional logics have uncountably many Post complete extensions for which closure under some, but not necessarily all, rules peculiar to the conditional are relaxed. Finally, we reflect on what our results tell us about the structure of certain lattices of conditional logics and their connections with modal logic.

\end{abstract}

\noindent
 \textbf{Keywords}: Conditional logic; Lattice theory; Modal logic; Post completeness; Selection function

    \section{Introduction}
    \label{Section:Introduction}

    Conditional logic studies and formalizes modes of reasoning involving various conditional expressions from natural language. The discipline came into its own with the important semantic work of   Lewis~\cite{Lewis1971CompletenessThree,Lewis1973Counterfactuals} and Stalnaker~\cite{Stalnaker1968}, who showed that many of these systems are amenable to a Kripke-style model theoretic analysis familiar from modal logic. 
    
    While the connections between conditional and modal logic are deep, and despite the fact that the study of structural features of large families of modal logics is now well-established, conditional logics generally have been investigated piecemeal rather than systematically. By contrast, many features of the lattice of normal modal logics are well-understood, and certain corners of the lattice have been entirely mapped out.\footnote{For a textbook presentation of many results in this area, consult Chagrov and Zakharyaschev~\cite{ChagrovZakharyaschev1997}.} For example, Scroggs~\cite{Scroggs1951Extensions} identified all of the (quasi-)normal extensions of \textsf{S5}, Makinson~\cite{Makinson1971} demonstrated (inter alia) that the lattice of normal modal logics has exactly two coatoms, and Segerberg~\cite{Segerberg1971,Segerberg1972PostCompleteness} extended the results of both in important ways. 

    In this article, we contribute to the systematic study of conditional logic through an examination of Post completeness. A logic is \textit{Post complete} if it is consistent but has no consistent proper extensions; relative to a certain lattice of extensions of a system, the Post complete systems are the \textit{coatoms} of the lattice (for further discussion, see Section~\ref{Subsection:Lattices of Conditional Logics} below). Our investigation will exploit connections with modal logic to extend some techniques and results of Makinson~\cite{Makinson1971} and Segerberg~\cite{Segerberg1972PostCompleteness} to conditional logic, but we will also prove certain results which serve to distinguish conditional logic from modal logic.

    The plan of the article is as follows. We present the basic conditional logics and families of conditional logics that we will examine, both axiomatically and semantically, in Section~\ref{Section:Logics}. In Section~\ref{Section:Post completeness}, we present the central results of the paper. We identify all of the finitely many regular and normal Post complete conditional logics, prove analogues of Makinson's embedding theorems for certain subfamilies of normal conditional logics, and show that there are uncountably many (non-normal) Post complete extensions of certain conditional logics. We also highlight some structural parallels and divergences between conditional and multimodal logic (a more detailed treatment is deferred to Appendix~\ref{Section:Post completeness in multimodal logic}). Finally, we conclude with some directions for future research in Section~\ref{Section:Conclusions}. 

    \section{Logics}
    \label{Section:Logics}

    In this section, we present the basic logics and families of logics we will be studying in this article. These are developed axiomatically in Section~\ref{Subsection:Axiom Systems} and then semantically (i.e., model theoretically) in Section~\ref{Subsection:Semantics}. First, however, we pin down some details concerning the language(s) we will be working with.

    Consider a propositional language $\mathcal{L}^\bot$ ($\mathcal{L}$) containing the propositional constant $\bot$ and the primitive connectives $\neg$, $\vee$, and $\boxright$ (and a denumerable set of variables $Var=\{p, q,\dots\}$). $\top, \wedge, \to$, and $\leftrightarrow$ are defined as usual (in particular, $\top$ is taken to be $\neg\bot$). We sometimes use $\square_{\varphi}\psi$ and $\lozenge_{\varphi}\psi$ as abbreviations for $(\varphi\boxright\psi)$ and $\neg\square_{\varphi}\neg\psi$, respectively (particularly, where the antecedent is $\top$; cf.~the ``inner modalities'' of Lewis~\cite[p.~30]{Lewis1973Counterfactuals}). Additionally, we use the following inductive definitions:
    \begin{itemize}
        \item[$i.$] $\square^0_\varphi\psi (\lozenge^0_\varphi\psi):=\psi$
        \item[$ii.$] $\square_\varphi^{n+1}\psi (\lozenge_\varphi^{n+1}\psi):=\square_\varphi\square^n_\varphi\psi (\lozenge_\varphi\lozenge^n_\varphi\psi)$
    \end{itemize} 
    Let $Fm^\bot$ ($Fm$) denote the set of formulas of $\mathcal{L}^\bot$ ($\mathcal{L}$). All of the logics we are interested in will be defined as subsets of $Fm$, but we will have occasion to make special technical use of $Fm^\bot$.

    A \textit{substitution} is a function $^\sigma:Var\to{Fm}$. Every substitution $^\sigma$ induces a function (which we write the same way) $^\sigma:Fm\to{Fm}$ which is a homomorphism. $\varphi^\sigma$ is the result of \textit{uniformly substituting} $p^\sigma$ for every occurrence of every variable $p$ in $\varphi$.

    \subsection{Axiom systems}
    \label{Subsection:Axiom Systems}

    In this section, we present a few important systems of conditional logic and families of conditional logics axiomatically. The conventions employed largely follow Chellas~\cite[\S4]{Chellas1975Basic} and Nute~\cite[pp.~128--130]{Nute1980Topics}, with some minor deviations and improvisations.
    
    \begin{definition}
    \label{Definition:ConditionalLogic}
        A \textit{conditional logic} \textsf{L} is a subset of $Fm$ containing all tautologies and closed under modus ponens and uniform substitution. 
    \end{definition}
    
    A formula $\varphi$ is a \textit{theorem} of the conditional logic \textsf{L} (in symbols, $\vdash_{\mathsf{L}}\varphi$) iff $\varphi\in\mathsf{L}$. $\varphi$ is \textsf{L}-\textit{derivable} from a set of formulas $\Gamma$ (in symbols, $\Gamma\vdash_{\mathsf{L}}\varphi$) iff for some $\gamma_1,\ldots,\gamma_n \in \Gamma$ ($n\geq0$), $\vdash_{\mathsf{L}}(\gamma_1\wedge\ldots\wedge\gamma_n)\to\varphi$.\footnote{Where $n=0$, $(\gamma_1\wedge\ldots\wedge\gamma_n)\to\varphi$ is taken to be $\varphi$.} A conditional logic \textsf{L} is \textit{consistent} iff $\not\vdash_{\mathsf{L}}\bot$. A set of formulas $\Gamma$ is \textsf{L}-\textit{consistent} iff $\Gamma\not\vdash_{\mathsf{L}}\bot$. Where \textsf{K} and \textsf{L} are conditional logics, we say that \textsf{L} is an \textit{extension} of \textsf{K} if $\textsf{K}\subseteq\textsf{L}$.

    From Definition~\ref{Definition:ConditionalLogic}, it is clear that the smallest conditional logic is simply classical propositional logic (\textsf{PL}) as formulated in $Fm$. The greatest conditional logic is the trivial logic, that is, $Fm$ itself. More interesting systems are obtained by augmenting \textsf{PL} by axioms or rules such as:
    \begin{equation*}
        \label{Axiom:CM}
        \tag{CM}
(\varphi\boxright(\psi\wedge\theta))\to((\varphi\boxright\psi)\wedge(\varphi\boxright\theta))
    \end{equation*}
    \begin{equation*}
        \label{Axiom:CC}
        \tag{CC}
((\varphi\boxright\psi)\wedge(\varphi\boxright\theta))\to(\varphi\boxright(\psi\wedge\theta))
    \end{equation*}
    \begin{equation*}
        \label{Axiom:CN}
        \tag{CN}
        \varphi\boxright\top
    \end{equation*}
    \begin{equation*}
        \label{Rule:RCEA}
        \tag{RCEA}
        \frac{\vdash\varphi\leftrightarrow\psi}{\vdash(\varphi\boxright\theta)\leftrightarrow(\psi\boxright\theta)}
    \end{equation*}
    \begin{equation*}
        \label{Rule:RCEC}
        \tag{RCEC}
        \frac{\vdash\varphi\leftrightarrow\psi}{\vdash(\theta\boxright\varphi)\leftrightarrow(\theta\boxright\psi)}
    \end{equation*}

    A conditional logic is \textit{semiregular} (\textit{regular}) if it contains \ref{Axiom:CM} and \ref{Axiom:CC} and is closed under \ref{Rule:RCEC} (and \ref{Rule:RCEA}). A conditional logic is \textit{seminormal} (\textit{normal}) if it is semiregular (regular) and it contains \ref{Axiom:CN}. The smallest semiregular, seminormal, regular, and normal conditional logics are called \textsf{Cr}, \textsf{Ck}, \textsf{CR}, and \textsf{CK}, respectively. A conditional logic is \textit{quasi-semiregular} (\textit{quasi-seminormal}, \textit{quasi-regular}, \textit{quasi-normal}) if it contains \textsf{Cr} (\textsf{Ck}, \textsf{CR}, \textsf{CK}).

    We will also make a few comments about two stronger families of conditional logics with axioms drawn from the following list:
    \begin{equation*}
        \label{Axiom:ID}
        \tag{ID}
        \varphi\boxright\varphi
    \end{equation*}
    \begin{equation*}
        \label{Axiom:MOD}
        \tag{MOD}
        (\neg\psi\boxright\bot)\to(\varphi\boxright\psi)
    \end{equation*}
    \begin{equation}
        \label{Axiom:SegerbergExportation}
        \tag{SE}
        ((\varphi\wedge\psi)\boxright\theta)\to(\varphi\boxright(\psi\to\theta))
    \end{equation}
    \begin{equation}
        \label{Axiom:SegerbergImport}
        \tag{SI}
        \neg(\varphi\boxright\neg\psi)\to((\varphi\boxright(\psi\to\theta))\to((\varphi\wedge\psi)\boxright\theta))
    \end{equation}
    \begin{equation}
        \label{Axiom:CMonotonicity}
        \tag{RM}
        ((\varphi\boxright\psi)\wedge(\varphi\boxright\theta))\to((\varphi\wedge\psi)\boxright\theta)
    \end{equation}

    \noindent A conditional logic is \textit{modalized} if it is normal and contains \ref{Axiom:ID} and \ref{Axiom:MOD} (Segerberg~\cite[p.~158]{Segerberg1989ConditionalNotes} reserves the word ``normal'' for such systems). The smallest modalized conditional logic is a system we call \textsf{M}. The hallmark of modalized systems is that $\square$ is a normal modal operator when it is defined as $\square\varphi:=\neg\varphi\boxright\bot$ (cf.~the ``outer modalities'' of Lewis~\cite[p.~22]{Lewis1973Counterfactuals}).\footnote{Consult \textcite[pp.~295--306]{Williamson2022philosophy} and \textcite[\S2.1.2]{Weiss2019Dissertation}.} Finally, a conditional logic is \textit{variably strict} if it is modalized and contains \ref{Axiom:SegerbergExportation}, \ref{Axiom:SegerbergImport}, and \ref{Axiom:CMonotonicity}. The smallest variably strict conditional logic is the system Lewis~\cite{Lewis1973Counterfactuals} calls \textsf{V}.\footnote{This is the same as the system Lewis calls \textsf{C0} in \cite{Lewis1971CompletenessThree}. The style of axiomatization here follows Segerberg~\cite{Segerberg1989ConditionalNotes}.}

    \subsection{Semantics}
    \label{Subsection:Semantics}

    We turn now to semantics (i.e., model theory). We consider two kinds of semantics suited to handling seminormal and normal conditional logics.\footnote{We do not develop a semantics for regular or semiregular conditional logics here, since we will not have occasion to employ it. The interested reader might consult the neighborhood (``minimal'') semantics for \textsf{CR} and extensions thereof from Chellas~\cite[\S8]{Chellas1975Basic}.} We develop both concurrently since both model theories employ ``selection functions'', and so belong to a semantic tradition in conditional logic going back to Stalnaker~\cite{Stalnaker1968}.\footnote{In the selection function semantics of Stalnaker~\cite{Stalnaker1968}, functions select \textit{worlds}; our selection functions are \textit{set} selection functions in the sense of Lewis~\cite[\S2.7]{Lewis1973Counterfactuals}.} They are distinguished by whether the selection functions are indexed by propositions (i.e., sets of worlds) or by formulas.\footnote{Chellas discusses both forms of model theory, but he only discusses the formula version in a note \cite[pp.~149--150, n.~14]{Chellas1975Basic}.}

    \begin{definition}
        \label{Definition:FormulaSelectionFrame}
        A \textit{formula selection frame} is a structure $\mathfrak{F}^{f}=\langle{W,f}\rangle$ where $W\neq\emptyset$ is a set of worlds and $f:W\times{Fm}\to\mathcal{P}(W)$ is a selection function indexed by worlds and formulas. 
    \end{definition}

    \begin{definition}
        \label{Definition:PropositionSelectionFrame}
        A \textit{proposition selection frame} is a structure $\mathfrak{F}^{\pi}=\langle{W,\pi}\rangle$ where $W\neq\emptyset$ is a set of worlds and $\pi:W\times\mathcal{P}(W)\to\mathcal{P}(W)$ is a selection function indexed by worlds and propositions. 
    \end{definition}

    \begin{definition}
        \label{Definition:Model}
        A \textit{formula} (\textit{proposition}) \textit{selection model} is a structure $\mathfrak{M}^{f}=\langle{W,f,V}\rangle$ ($\mathfrak{M}^{\pi}=\langle{W,\pi,V}\rangle$), where $\mathfrak{F}^{f}=\langle{W,f}\rangle$ ($\mathfrak{F}^{\pi}=\langle{W,\pi}\rangle$) is a formula (proposition) selection frame and $V:Var\to\mathcal{P}(W)$ is a valuation.
    \end{definition}

    With respect to a formula selection model $\mathfrak{M}^f = \langle W, f, V \rangle$ and world $w\in W$, the relation $\models^{\mathfrak{M}^f}_{w}$ is defined as follows:
    \begin{itemize}
        \item[$p$.] $\models^{\mathfrak{M}^{f}}_{w}p$ iff $w\in V(p)$;
        \item[$\bot$.] $\models^{\mathfrak{M}^{f}}_{w}\bot$ never;
        \item[$\neg$.] $\models^{\mathfrak{M}^{f}}_{w}\neg\varphi$ iff $\not\models^{\mathfrak{M}^{f}}_{w}\varphi$;
        \item[$\vee$.] $\models^{\mathfrak{M}^{f}}_{w}\varphi\vee\psi$ iff $\models^{\mathfrak{M}^{f}}_{w}\varphi$ or $\models^{\mathfrak{M}^{f}}_{w}\psi$;
        \item[$\boxright$.] $\models^{\mathfrak{M}^{f}}_{w}\varphi\boxright\psi$ iff $f(w, \varphi)\subseteq[\psi]^{\mathfrak{M}^{f}}$.
    \end{itemize}
    With respect to a proposition selection model $\mathfrak{M}^\pi = \langle W, \pi, V \rangle$ and world $w\in W$, the relation $\models^{\mathfrak{M}^\pi}_{w}$ is defined in the same way except:
    \begin{itemize}    
        \item[$\boxright'$.] $\models^{\mathfrak{M}^{\pi}}_{w}\varphi\boxright\psi$ iff $\pi(w,[\varphi]^{\mathfrak{M}^{\pi}})\subseteq[\psi]^{\mathfrak{M}^{\pi}}$.
    \end{itemize} 
    We define that $\models^{\mathfrak{M}^{f}}\varphi$ iff $\models^{\mathfrak{M}^{f}}_{w}\varphi$ for all $w\in W$; and $\models^{\mathfrak{F}^{f}}\varphi$ iff $\models^{\mathfrak{M}^{f}}\varphi$ for every model $\mathfrak{M}^{f}$ over $\mathfrak{F}^{f}$. For a formula selection frame $\mathfrak{F}^{f}$, by $\mathsf{L}(\mathfrak{F}^{f})$ is intended $\{\varphi\in{Fm}:\models^{\mathfrak{F}^{f}}\varphi\}$;\footnote{Caution: for formula selection frames $\mathfrak{F}^f$, $\mathsf{L}(\mathfrak{F}^f)$ is not, in general, a logic. In particular, $\mathsf{L}(\mathfrak{F}^f)$ may fail to be closed under uniform substitution.} and for a class of such frames $\mathcal{C}^{f}$, by $\mathsf{L}(\mathcal{C}^{f})$ is intended $\bigcap_{\mathfrak{F}^{f}\in\mathcal{C}^{f}}\mathsf{L}(\mathfrak{F}^{f})$. The definitions of $\mathsf{L}(\mathfrak{F}^{\pi})$, $\mathsf{L}(\mathcal{C}^{\pi})$, etc., are the same (mutatis mutandis). 

    \begin{definition}
        \label{Definition:Validity}
        A formula $\varphi$ is \textit{valid} in a formula (proposition) selection frame $\mathfrak{F}^{f}$ ($\mathfrak{F}^{\pi}$) iff $\varphi\in\mathsf{L}(\mathfrak{F}^{f})$ ($\varphi\in\mathsf{L}(\mathfrak{F}^{\pi})$). A formula $\varphi$ is \textit{valid} in a class of formula (proposition) selection frames $\mathcal{C}^{f}$ ($\mathcal{C}^{\pi}$) iff $\varphi\in\mathsf{L}(\mathcal{C}^{f})$ ($\varphi\in\mathsf{L}(\mathcal{C}^{\pi})$).
    \end{definition}

     Hereafter we sometimes omit superscripts (e.g., writing $\mathfrak{M}$ for $\mathfrak{M}^\pi$) where context disambiguates. Chellas~\cite[\S6]{Chellas1975Basic} proved that \textsf{CK} is sound and complete with respect to the proposition selection semantics (i.e., $\vdash_{\mathsf{CK}}\varphi$ iff $\varphi\in\mathsf{L}(\mathcal{PS})$, where $\mathcal{PS}$ is the class of all proposition selection frames). For a proof that \textsf{Ck} is sound and complete with respect to the formula selection function semantics, consult Weiss~\cite[\S2.3.1]{Weiss2019Dissertation}.

     The basic modalized and variably strict systems, \textsf{M} and \textsf{V}, can be semantically characterized by (i.e., are sound and complete with respect to) certain subclasses of proposition selection frames. For instance, \textsf{M} is characterized by the class of all proposition selection frames $\langle W, \pi \rangle$ satisfying the following two conditions ($w \in W$; $X, Y \subseteq W$):
    \begin{equation*}
        \label{Condition:id}
        \tag{id}
        \pi(w,X)\subseteq X
    \end{equation*}
    \begin{equation*}
        \label{Condition:mod}
        \tag{mod}
        \pi(w,X)=\emptyset \Rightarrow \pi(w,Y)\cap X = \emptyset
    \end{equation*}
    This result is (implicitly) proved by Segerberg~\cite{Segerberg1989ConditionalNotes}, who proves completeness for an array of conditional logics with respect to classes of general frames and subsequently proves the finite model property for them. 

    The following semantic result, whose straightforward proof we omit, will be useful later in the paper.
    
    \begin{lemma}
        \label{Lemma:PropFrameInducesaNormalLogic}
        For any non-empty class $\mathcal{C}$ of proposition selection frames, $\mathsf{L}(\mathcal{C})$ is a consistent normal conditional logic.
    \end{lemma}

    \section{Post completeness}
    \label{Section:Post completeness}

    In this section, we present our main results on Post completeness in conditional logic. First, some definitions:

    \begin{definition}
        \label{Definition:PostComplete}
        A conditional logic \textsf{L} is \textit{Post complete} if it is consistent but has no consistent proper extensions (equivalently, whose only proper extension is $Fm$).\footnote{Post~\cite[p.~177]{Post1921ElementaryPropositions} calls such systems ``closed'' (cf.~Tarski~\cite[p.~93]{Tarski1930Methodology}).}
    \end{definition}

    \begin{definition}
        \label{Definition:PostNumber}
        The \textit{Post number} of a conditional logic \textsf{L}, denoted $\rho(\mathsf{L})$, is the number of Post complete extensions of \textsf{L}.
    \end{definition}

    A couple of abstract and general results allow us to put weak bounds on $\rho$ (we omit the proof of the second, but see Tarski~\cite[p.~98]{Tarski1930Methodology}; cf.~Segerberg~\cite[p.~711]{Segerberg1972PostCompleteness}):

    \begin{lemma}
      \label{Lemma:MaxPostCompleteExtensions}
      For any conditional logic $\mathsf{L}$, $\rho(\mathsf{L})\leq2^{\aleph_0}$.
    \end{lemma}
    \begin{proof}
      A conditional logic is a subset of $Fm$ (Definition~\ref{Definition:ConditionalLogic}), and there are clearly $2^{\aleph_0}$ subsets of $Fm$.
    \end{proof}

    \begin{lemma}[Lindenbaum]
        \label{Lemma:Lindenbaum}
        For any consistent conditional logic \emph{\textsf{L}}, $\rho(\mathsf{L})\geq1$.
    \end{lemma}

    The plan of the section is as follows. In Section~\ref{Subsection:Lattices of Conditional Logics}, we describe some general lattice theoretical features of conditional logics and examine Post completeness from this perspective. In Section~\ref{Subsection:Regular and Normal Post complete extensions}, we identify all of the finitely many regular and normal Post complete conditional logics. For consistent modalized and variably strict conditional logics, we also prove analogues of Makinson's embedding theorems from \cite{Makinson1971} and show that all such logics have proposition selection frames. Finally, in Section~\ref{Subsection:Further Post complete extensions}, we show that various basic conditional logics have uncountably many Post complete extensions.

    \subsection{Lattices of conditional logics}
    \label{Subsection:Lattices of Conditional Logics}

    It is clear that the extensions of any conditional logic \textsf{L}, $\Ext(\mathsf{L})$, form a complete lattice ordered by inclusion ($\subseteq$), in which intersection ($\cap$) is meet and closure of unions under modus ponens and uniform substitution ($+$) is join. By $\mathsf{L}+X$ is intended the smallest extension of $\mathsf{L}$ containing the axiom $X$ (or, more generally, axioms from the set $X$), that is, the smallest set containing $\mathsf{L}$ and $X$ which is closed under modus ponens and uniform substitution.

    We may further consider the lattice of all semiregular (seminormal) extensions of a semiregular (seminormal) conditional logic \textsf{L}, $\rExt(\mathsf{L})$ ($\nExt(\mathsf{L})$); this is again a complete lattice in which intersection is meet and closure of unions under modus ponens, uniform substitution, and \ref{Rule:RCEC} ($\boxplus$) is join. Where \textsf{L} is a semiregular (seminormal) conditional logic, by $\mathsf{L}\boxplus X$ is intended the smallest semiregular (seminormal) extension of $\mathsf{L}$ containing $X$. 
    
    Finally, we may consider the lattice of all regular (normal) extensions of a regular (normal) conditional logic \textsf{L}, $\RExt(\mathsf{L})$ ($\NExt(\mathsf{L})$); this is, once more, a complete lattice in which intersection is meet and closure of unions under modus ponens, uniform substitution, \ref{Rule:RCEC}, and \ref{Rule:RCEA} ($\oplus$) is join. Where \textsf{L} is a regular (normal) conditional logic, by $\mathsf{L}\oplus X$ is intended the smallest regular (normal) extension of $\mathsf{L}$ containing $X$.

    It is clear that the bottom element in the lattice $\Ext(\mathsf{L})$ ($\rExt(\mathsf{L})$, etc.) is \textsf{L} and the top element is $Fm$. By a \textit{coatom}, we mean an element $\mathsf{K}$ of the lattice such that $\mathsf{K}\subset Fm$ but for any $\mathsf{K}'\supseteq\mathsf{K}$, $\mathsf{K}'=\mathsf{K}$ or $\mathsf{K}'=Fm$. A lattice is \textit{coatomic} if for any $\mathsf{K}\subset Fm$, $\mathsf{K}$ is contained in a coatom.

    The relationship between being a coatom and being a Post complete extension is given by the following result.

    \begin{lemma}
        \label{lem:post=coatoms}
        For any conditional logic $\mathsf{L}$, the coatoms of $\Ext(\mathsf{L})$ are exactly the Post complete extensions of $\mathsf{L}$.
    \end{lemma}

    \begin{proof}
        That each Post complete extension of $\mathsf{L}$ is a coatom of $\Ext(\mathsf{L})$ is a direct consequence of Definition~\ref{Definition:PostComplete}. For the converse direction, assume for reductio that a coatom $\mathsf{C}$ of $\Ext(\mathsf{L})$ is not Post complete; by Lemma \ref{Lemma:Lindenbaum}, $\mathsf{C}$ has a Post complete extension $\mathsf{C}'$. Therefore, $\mathsf{L} \subseteq \mathsf{C} \subset \mathsf{C}' \subset Fm$ and $\mathsf{C}' \in \Ext(\mathsf{L})$, contradicting the fact that $\mathsf{C}$ is a coatom in $\Ext(\mathsf{L})$.
    \end{proof}

    In view of Lemma~\ref{lem:post=coatoms}, $\rho(\textsf{L})$ may be calculated as the number of coatoms in $\Ext(\mathsf{L})$. It is important to emphasize that the Post complete extensions of a given logic need not coincide with the coatoms in more restrictive lattices of extensions. We will shortly observe that for lattices such as $\NExt(\mathsf{L})$, the number of coatoms and the number of Post complete extensions \textit{in} $\NExt(\mathsf{L})$ can come apart dramatically. 
    
    It follows from Lemmas~\ref{Lemma:Lindenbaum} and \ref{lem:post=coatoms} that the lattice $\Ext(\mathsf{L})$ is coatomic for any conditional logic \textsf{L}. This also follows from a more general result whose proof we now rehearse. Let $\mathcal{A}\subseteq Fm$ and let $\mathcal{R} \subseteq \{\text{\ref{Rule:RCEA}, \ref{Rule:RCEC}}\}$ be a collection of inference rules. Let $\mathsf{L}^\mathcal{R}_{\mathcal{A}}$ be the smallest conditional logic containing $\mathcal{A}$ and closed under the rules in $\mathcal{R}$. Let $\mathcal{R}\mathsf{Ext}(\mathsf{L}^\mathcal{R}_\mathcal{A})$ be the lattice of extensions of $\mathsf{L}^\mathcal{R}_\mathcal{A}$ closed under the rules in $\mathcal{R}$. Namely, $\mathcal{R}\mathsf{Ext}(\mathsf{L}^\mathcal{R}_\mathcal{A})$ is the lattice of all the conditional logics containing $\mathcal{A}$ and closed under the rules in $\mathcal{R}$. 

    \begin{theorem}
        \label{Theorem:Coatomicity}
       For arbitrary $\mathcal{A}\subseteq Fm$ and $\mathcal{R} \subseteq \{\text{\em \ref{Rule:RCEA}, \ref{Rule:RCEC}}\}$, the lattice $\mathcal{R}\mathsf{Ext}(\mathsf{L}^\mathcal{R}_\mathcal{A})$ is coatomic.
    \end{theorem}
    \begin{proof}
        If $\mathsf{L}^\mathcal{R}_\mathcal{A}$ is inconsistent, then $\mathcal{R}\mathsf{Ext}(\mathsf{L}^\mathcal{R}_\mathcal{A})$ contains only one element, hence it is vacuously coatomic. Otherwise, if $\mathsf{L}^\mathcal{R}_\mathcal{A}$ is consistent, take any consistent $\mathsf{L} \in \mathcal{R}\mathsf{Ext}(\mathsf{L}^\mathcal{R}_\mathcal{A})$ and consider the poset of consistent extensions of $\mathsf{L}$ in $\mathcal{R}\mathsf{Ext}(\mathsf{L}^\mathcal{R}_\mathcal{A})$
        \[
            \mathsf{Poset}^\mathcal{R}_\mathcal{A}(\mathsf{L})=\{\mathsf{L}'\in\mathcal{R}\mathsf{Ext}(\mathsf{L}^\mathcal{R}_\mathcal{A}):
            \mathsf{L}\subseteq\mathsf{L}'\text{ and }\mathsf{L}'\text{ is consistent}\}
        \]
        ordered by $\subseteq$. Clearly $\mathsf{Poset}^\mathcal{R}_\mathcal{A}(\mathsf{L})$ is nonempty since $\mathsf{L} \in \mathsf{Poset}^\mathcal{R}_\mathcal{A}(\mathsf{L})$. Now, consider a nonempty chain $\mathcal{C} \subseteq \mathsf{Poset}^\mathcal{R}_\mathcal{A}(\mathsf{L})$, namely for all $\mathsf{L}', \mathsf{L}'' \in \mathcal{C}$, either $\mathsf{L}' \subseteq \mathsf{L}''$ or $\mathsf{L}'' \subseteq \mathsf{L}'$. Now consider
        \[
            \mathsf{U}=\bigcup_{\mathsf{L}\in\mathcal{C}}\mathsf{L}.
        \]
        Since $\mathcal{C}\neq\emptyset$, $\mathsf{U}$ contains $\mathcal{A}$ and all classical propositional tautologies. Moreover, $\mathsf{U}$ is closed under uniform substitution; indeed, assume $\varphi \in \mathsf{U}$, then $\varphi \in \mathsf{L}'$ for some $\mathsf{L}' \in \mathcal{C}$. Since $\mathsf{L}'$ is a conditional logic, then $\varphi^\sigma \in \mathsf{L}'$, hence $\varphi^\sigma \in \mathsf{U}$ for all substitutions $^\sigma$. $\mathsf{U}$ is also closed under modus ponens; indeed, assume $\varphi, \varphi \to \psi \in \mathsf{U}$. This means that $\varphi \in \mathsf{L}'$ and $\varphi \to \psi \in \mathsf{L}''$ for some $\mathsf{L}', \mathsf{L}'' \in \mathcal{C}$. Since $\mathsf{C}$ is a chain, then either $\mathsf{L}' \subseteq \mathsf{L}''$ or $\mathsf{L}'' \subseteq \mathsf{L}'$. W.l.o.g. assume the former, then $\varphi, \varphi \to \psi \in \mathsf{L}''$ and, since $\mathsf{L}''$ is a conditional logic (hence closed under modus ponens), we have that $\psi \in \mathsf{L}''$, therefore, $\psi \in \mathsf{U}$ as well.

        Now, we show that $\mathsf{U}$ is closed under all the rules in $\mathcal{R}$. Assume $\varphi \leftrightarrow \psi \in \mathsf{U}$, this means that $\varphi \leftrightarrow \psi \in \mathsf{L}'$ for some $\mathsf{L}' \in \mathcal{C}$; since $\mathcal{L}' \in \mathcal{R}\mathsf{Ext}(\mathsf{L}^\mathcal{R}_\mathcal{A})$, then $\mathsf{L}'$ is closed under the rules in $\mathcal{R}$. Therefore, $(i)$ $(\delta \boxright \varphi) \leftrightarrow (\delta \boxright \psi) \in\mathsf{L}'$, or $(ii)$ $(\varphi \boxright \delta) \leftrightarrow (\psi \boxright \delta) \in\mathsf{L}'$, or both. Hence, in each case we have, respectively, that $(\delta \boxright \varphi) \leftrightarrow (\delta \boxright \psi) \in\mathsf{U}$, or $(\varphi \boxright \delta) \leftrightarrow (\psi \boxright \delta) \in\mathsf{U}$, or both. Finally, $\mathsf{U}$ is consistent, otherwise $\bot \in \mathsf{L}'$ for some $\mathsf{L}' \in \mathcal{C}$ contradicting the fact that $\mathcal{C}$ contains only consistent elements, by definition.

        Therefore, $\mathsf{U}\in\mathsf{Poset}^\mathcal{R}_\mathcal{A}(\mathsf{L})$ and, clearly, $\mathsf{U}$ is an upper bound of $\mathcal{C}$. The upper bound of $\emptyset$ in $\mathsf{Poset}^\mathcal{R}_\mathcal{A}(\mathsf{L})$ is $\mathsf{L}$ itself.

        Now, by Zorn's lemma, $\mathsf{Poset}^\mathcal{R}_\mathcal{A}(\mathsf{L})$ has a maximal member $\mathsf{P}$. Observe that $\mathsf{P} \in \mathcal{R}\mathsf{Ext}(\mathsf{L}^\mathcal{R}_\mathcal{A})$. Now, assume that $\mathsf{P} \subset \mathsf{P}'$ for some $\mathsf{P}' \in \mathcal{R}\mathsf{Ext}(\mathsf{L}^\mathcal{R}_\mathcal{A})$ such that $\mathsf{P}'$ is consistent. We have that $\mathsf{P}' \in \mathsf{Poset}^\mathcal{R}_\mathcal{A}(\mathsf{L})$, contradicting the fact that $\mathsf{P}$ is a maximal member. Therefore, $\mathsf{P}$ is a coatom of  $\mathcal{R}\mathsf{Ext}(\mathsf{L}^\mathcal{R}_\mathcal{A})$ and $\mathsf{L} \subseteq \mathsf{P}$.

        Since $\mathsf{L}$ was an arbitrary consistent logic in $\mathcal{R}\mathsf{Ext}(\mathsf{L}^\mathcal{R}_\mathcal{A})$, we have that $\mathcal{R}\mathsf{Ext}(\mathsf{L}^\mathcal{R}_\mathcal{A})$ is coatomic.
    \end{proof}

    As a corollary of Theorem~\ref{Theorem:Coatomicity} we obtain the following.

    \begin{corollary}
        \label{Corollary:Coatomicity}

        For a conditional logic $\mathsf{L}$, the following hold:
        
        \begin{itemize}
            \item[$i.$] $\Ext(\mathsf{L})$ is coatomic;

            \item[$ii.$] if $\mathsf{L}$ is regular (semiregular), then $\RExt(\mathsf{L})$ ($\rExt(\mathsf{L})$) is coatomic;

            \item[$iii.$] if $\mathsf{L}$ is normal (seminormal), then $\NExt(\mathsf{L})$ ($\nExt(\mathsf{L})$) is coatomic.
        \end{itemize}
    \end{corollary}

    Note, in particular, that $\rExt(\mathsf{Cr})$, $\nExt(\mathsf{Ck})$, $\RExt(\mathsf{CR})$, and $\NExt(\mathsf{CK})$ are coatomic.
    
    \subsection{Regular and normal Post complete extensions}
    \label{Subsection:Regular and Normal Post complete extensions}

    Our investigation in this section will make heavy use of the following special case of a general theorem of Tarski's from \cite[p.~395]{Tarski1934-35Extensions}, previously singled out for employment in modal logic by Segerberg~\cite[p.~712]{Segerberg1972PostCompleteness}:
    
    \begin{lemma}[Segerberg/Tarski]
    \label{Lemma:SegerbergA}
    For any consistent conditional logic \emph{\textsf{L}}, $\rho(\mathsf{L})=1$ iff for all $\varphi\in{Fm}^\bot$, $\vdash_{\mathsf{L}}\varphi$ or $\vdash_{\mathsf{L}}\neg\varphi$.
    \end{lemma}

    The first of our principal theorems in this section (Theorem~\ref{Theorem:RegularPCLogics}) is that there are exactly nine Post complete regular conditional logics. They arise from the following incompatible subsets of $Fm^\bot$:
    \begin{equation*}
        \label{RegularSet:Top}
        \tag{top}
        \{\top\boxright\top,\top\boxright\bot,\bot\boxright\top,\bot\boxright\bot\}
    \end{equation*}
    \begin{equation*}
        \label{RegularSet:DIS}
        \tag{disj}
        \{\top\boxright\top,\top\boxright\bot,\bot\boxright\top,\neg(\bot\boxright\bot)\}
    \end{equation*}
    \begin{equation*}
        \label{RegularSet:Antecedent}
        \tag{ant}
        \{\top\boxright\top,\top\boxright\bot,\neg(\bot\boxright\top),\neg(\bot\boxright\bot)\}
    \end{equation*}
    \begin{equation*}
        \label{RegularSet:IMP}
        \tag{imp}
        \{\top\boxright\top,\neg(\top\boxright\bot),\bot\boxright\top,\bot\boxright\bot\}
    \end{equation*}
    \begin{equation*}
        \label{RegularSet:Consequent}
        \tag{con}
        \{\top\boxright\top,\neg(\top\boxright\bot),\bot\boxright\top,\neg(\bot\boxright\bot)\}
    \end{equation*}
    \begin{equation*}
        \label{RegularSet:CON}
        \tag{conj}
        \{\top\boxright\top,\neg(\top\boxright\bot),\neg(\bot\boxright\top),\neg(\bot\boxright\bot)\}
    \end{equation*}
    \begin{equation*}
        \label{RegularSet:NegatedAntecedent}
        \tag{nant}
        \{\neg(\top\boxright\top),\neg(\top\boxright\bot),\bot\boxright\top,\bot\boxright\bot\}
    \end{equation*}
    \begin{equation*}
        \label{RegularSet:NegatedConjunction}
        \tag{nconj}
        \{\neg(\top\boxright\top),\neg(\top\boxright\bot),\bot\boxright\top,\neg(\bot\boxright\bot)\}
    \end{equation*}
    \begin{equation*}
        \label{RegularSet:Bottom}
        \tag{bot}
        \{\neg(\top\boxright\top),\neg(\top\boxright\bot),\neg(\bot\boxright\top),\neg(\bot\boxright\bot)\}
    \end{equation*}
    Essentially, these are all of the ways of accepting or rejecting a conditional holding between $\bot$ and $\top$ that are consistent with \textsf{CR}.\footnote{Note that $\vdash_{\mathsf{CR}}(\bot\boxright\bot)\to(\bot\boxright\top)$ and $\vdash_{\mathsf{CR}}(\top\boxright\bot)\to(\top\boxright\top)$. Incidentally, we focus on the regular conditional logics partly to reinforce the analogy with Segerberg~\cite{Segerberg1972PostCompleteness}, and partly because we regard this as the largest class of conditional logics which imposes some sort of interesting logical structure on systems. However, it would be easy to generalize some of our results to wider classes of conditional logics (e.g., classical conditional logics; see Chellas~\cite[\S4]{Chellas1975Basic}) which recognize more consistent subsets of $Fm^\bot$.} Of these sets, only (\ref{RegularSet:Top}), (\ref{RegularSet:DIS}), (\ref{RegularSet:IMP}), and (\ref{RegularSet:Consequent}) are consistent with \textsf{CK}.

    \begin{lemma}
        \label{Lemma:RegularSets}
        If $\mathsf{L}$ is a regular conditional logic containing any of \emph{(\ref{RegularSet:Top})}--\emph{(\ref{RegularSet:Bottom})}, then for all $\varphi\in{Fm}^\bot$, $\vdash_{\mathsf{L}}\varphi$ or $\vdash_{\mathsf{L}}\neg\varphi$.
    \end{lemma}
    \begin{proof}
        By induction on the complexity of $\varphi$. The cases where $\varphi$ is of the form $\bot$, $\neg\psi$, or $\psi\vee\theta$ are trivial. When it is of the form $\psi\boxright\theta$, the argument with respect to each set proceeds by considering subcases according to whether $\vdash_{\mathsf{L}}\psi$ and $\vdash_{\mathsf{L}}\theta$, and making appropriate substitutions using \ref{Rule:RCEA} and \ref{Rule:RCEC}. For example, in the case of (\ref{RegularSet:NegatedAntecedent}), if $\vdash_{\mathsf{L}}\psi$ and $\vdash_{\mathsf{L}}\theta$, then $\vdash_{\mathsf{L}}\neg(\psi\boxright\theta)$ since $\vdash_{\mathsf{L}}\neg(\top\boxright\top)$; if $\vdash_{\mathsf{L}}\psi$ and $\not\vdash_{\mathsf{L}}\theta$, then by the induction hypothesis, $\vdash_{\mathsf{L}}\neg\theta$, whence $\vdash_{\mathsf{L}}\neg(\psi\boxright\theta)$ since $\vdash_{\mathsf{L}}\neg(\top\boxright\bot)$; and in the other cases, by analogous reasoning, $\vdash_{\mathsf{L}}\psi\boxright\theta$.
    \end{proof}

    %{\color{blue} Actually, it seems obvious that the foregoing result could be extended to any conditional logic merely closed under \ref{Rule:RCEA} and \ref{Rule:RCEC}; regularity is just paring down the number of possible options.}

    We now describe all of the putative Post complete regular conditional logics. Their characteristic axioms are as follows:
    \begin{equation*}
        \label{Axiom:Top}
        \tag{TOP}
        (\varphi\boxright\psi)\leftrightarrow\top
    \end{equation*}
    \begin{equation*}
        \label{Axiom:Disjunction}
        \tag{DISJ}
        (\varphi\boxright\psi)\leftrightarrow(\varphi\vee\psi)
    \end{equation*}
    \begin{equation*}
        \label{Axiom:Antecedent}
        \tag{ANT}
        (\varphi\boxright\psi)\leftrightarrow\varphi
    \end{equation*}
    \begin{equation*}
        \label{Axiom:Imp}
        \tag{IMP}
        (\varphi\boxright\psi)\leftrightarrow(\varphi\to\psi)
    \end{equation*}
    \begin{equation*}
        \label{Axiom:Consequent}
        \tag{CON}
        (\varphi\boxright\psi)\leftrightarrow\psi
    \end{equation*}
    \begin{equation*}
        \label{Axiom:Conjunction}
        \tag{CONJ}
        (\varphi\boxright\psi)\leftrightarrow(\varphi\wedge\psi)
    \end{equation*}
    \begin{equation*}
        \label{Axiom:NegatedAntecedent}
        \tag{NANT}
        (\varphi\boxright\psi)\leftrightarrow\neg\varphi
    \end{equation*}
    \begin{equation*}
        \label{Axiom:NegatedConjunction}
        \tag{NCONJ}
        (\varphi\boxright\psi)\leftrightarrow(\neg\varphi\wedge\psi)
    \end{equation*}
    \begin{equation*}
        \label{Axiom:Bot}
        \tag{BOT}
        (\varphi\boxright\psi)\leftrightarrow\bot
    \end{equation*}

    \begin{lemma}
        \label{Lemma:RegularPCLogics}
        Where $X$ is any of the immediately previously listed nine axiom schemata, $\mathsf{CR}\oplus{X}$ is a Post complete regular conditional logic.
    \end{lemma}
    \begin{proof}
        It must be proved that each system $\textsf{CR}\oplus{X}$ is consistent but has no consistent proper extensions. Both claims follow by showing translation results for each system $\textsf{CR}\oplus{X}$. The translation translates formulae of $Fm$ into the $\boxright$-free language. The translation varies by system (which we individuate by the characteristic axiom $X$):
        \begin{itemize}
            \item[$p$.] $p^{\tau(X)} = p$;
            \item[$\bot$.] $\bot^{\tau(X)} = \bot$;
            \item[$\neg$.] $(\neg\varphi)^{\tau(X)} = \neg\varphi^{\tau(X)}$;
            \item[$\vee$.] $(\varphi\vee\psi)^{\tau(X)} = \varphi^{\tau(X)} \vee \psi^{\tau(X)}$;
            \item[$\boxright$.] $(\varphi\boxright\psi)^{\tau(X)} = RHS(X)^{\tau(X)}$.
        \end{itemize}
        By $RHS(X)^{\tau(X)}$ is intended the translation of the right-hand side of the biconditional in the axiom $X$. For example, for $\textsf{CR}\oplus{\text{\ref{Axiom:Conjunction}}}$, $(\varphi\boxright\psi)^{\tau(\text{\ref{Axiom:Conjunction}})} = (\varphi\wedge\psi)^{\tau(\text{\ref{Axiom:Conjunction}})} = \varphi^{\tau(\text{\ref{Axiom:Conjunction}})}\wedge\psi^{\tau(\text{\ref{Axiom:Conjunction}})}$. 
        
        To show consistency, it is shown by induction on the length of proof that $\vdash_{\textsf{CR}\oplus{X}}\varphi$ implies $\vdash_{\mathsf{PL}}\varphi^{\tau(X)}$. It follows that $\textsf{CR}\oplus{X}$ is consistent (else \textsf{PL} would be inconsistent). To show Post completeness, it is shown by induction that $\vdash_{\textsf{CR}\oplus{X}}\varphi$ iff $\vdash_{\textsf{CR}\oplus{X}}\varphi^{\tau(X)}$. Now suppose for contradiction that $\textsf{CR}\oplus{X}$ had a consistent proper extension; then for some formula $\psi$, $\not\vdash_{\textsf{CR}\oplus{X}}\psi$ but $(\mathsf{CR}\oplus{X})+\psi$ is consistent, whence $(\mathsf{CR}\oplus{X})+\psi^{\tau(X)}$ is consistent. Then $\not\vdash_{\textsf{CR}\oplus{X}}\psi^{\tau(X)}$, whence $\psi^{\tau(X)}\not\in\textsf{PL}$. But \textsf{PL} is Post complete (in the $\boxright$-free language), so $\textsf{PL}+\psi^{\tau(X)}$ is inconsistent, and therefore $(\mathsf{CR}\oplus{X})+\psi^{\tau(X)}$ must be inconsistent.
    \end{proof}

\noindent
     As the formulas in each set $(x)$ are derivable from the corresponding axiom $X$ in $\mathsf{CR} \oplus X$, we obtain the following lemma.

    \begin{lemma}
        \label{Lemma:RegularSetsInclusion}
        For each set labelled ($x$) and corresponding axiom $X$, $(x)\subseteq\mathsf{CR}\oplus{X}$.
    \end{lemma}

    \begin{theorem}
        \label{Theorem:RegularPCLogics}
        There are exactly nine Post complete regular conditional logics, and they are the systems $\mathsf{CR}\oplus{X}$, for the axioms $X$ enumerated above.
    \end{theorem}
    \begin{proof}
        That there must be at least nine Post complete regular conditional logics follows from Lemmas~\ref{Lemma:RegularPCLogics} and \ref{Lemma:RegularSetsInclusion}, noting that each of the nine sets are logically incompatible with one another.

        To show that these are the only Post complete regular conditional logics, consider any Post complete conditional logic $\textsf{L}\in\RExt(\mathsf{CR})$. Clearly, $\rho(\mathsf{L})=1$, so by Lemma~\ref{Lemma:SegerbergA}, $\vdash_{\mathsf{L}}\top\boxright\top$ or $\vdash_{\mathsf{L}}\neg(\top\boxright\top)$; $\vdash_{\mathsf{L}}\top\boxright\bot$ or $\vdash_{\mathsf{L}}\neg(\top\boxright\bot)$; $\vdash_{\mathsf{L}}\bot\boxright\top$ or $\vdash_{\mathsf{L}}\neg(\bot\boxright\top)$; and $\vdash_{\mathsf{L}}\bot\boxright\bot$ or $\vdash_{\mathsf{L}}\neg(\bot\boxright\bot)$. Certain combinations are impossible in any regular conditional logic (e.g., $\vdash_{\mathsf{L}}\bot\boxright\bot$ and $\vdash_{\mathsf{L}}\neg(\bot\boxright\top)$). The possible combinations are those enumerated in the labelled sets, whence it is clear that \textsf{L} must contain one of the sets $(x)$. Then \textsf{L} and $\textsf{CR}\oplus{X}$ are both Post complete extensions of $\textsf{CR}\oplus(x)$; by Lemmas~\ref{Lemma:SegerbergA} and \ref{Lemma:RegularSets}, $\textsf{L}=\textsf{CR}\oplus{X}$.
    \end{proof}

    \begin{corollary}
        \label{Corollary:NormalPCLogics}
        There are exactly four Post complete normal conditional logics: $\mathsf{CR}\oplus{\text{\em \ref{Axiom:Top}}}$, $\mathsf{CR}\oplus{\text{\em \ref{Axiom:Disjunction}}}$, $\mathsf{CR}\oplus{\text{\em \ref{Axiom:Imp}}}$, and $\mathsf{CR}\oplus{\text{\em \ref{Axiom:Consequent}}}$.
    \end{corollary}
    \begin{proof}
        It is easy to show that each of these four systems contains \ref{Axiom:CN}, and therefore is normal. It is also easy to show that all of the other Post complete regular conditional logics cannot be consistently extended by \ref{Axiom:CN}, and hence cannot be normal. 
    \end{proof}

    \begin{corollary}
        \label{Corollary:ModalizedPCLogics}
        There are exactly two Post complete modalized conditional logics: $\mathsf{CR}\oplus{\text{\em \ref{Axiom:Top}}}$ and  $\mathsf{CR}\oplus{\text{\em \ref{Axiom:Imp}}}$.
    \end{corollary}
    \begin{proof}
        It is straightforward to show that these two systems are modalized, and to show that neither $\mathsf{CR}\oplus{\text{\ref{Axiom:Disjunction}}}$ nor $\mathsf{CR}\oplus{\text{\ref{Axiom:Consequent}}}$ can consistently be extended by \ref{Axiom:ID}.
    \end{proof}

    In fact, it is clear that $\mathsf{CR}\oplus{\text{\ref{Axiom:Top}}}$ and  $\mathsf{CR}\oplus{\text{\ref{Axiom:Imp}}}$ are both variably strict conditional logics, whence:

    \begin{corollary}
        \label{Corollary:NormalVSLogics}
        There are exactly two Post complete variably strict conditional logics: $\mathsf{CR}\oplus{\text{\em \ref{Axiom:Top}}}$ and  $\mathsf{CR}\oplus{\text{\em \ref{Axiom:Imp}}}$.
    \end{corollary}

    One of our normal Post complete conditional logics can be found in the literature on conditional logic already. $\mathsf{CR}\oplus{\text{\ref{Axiom:Imp}}}$ is equivalent to the system \textsf{VCA} of Lewis~\cite[pp.~130--131]{Lewis1973Counterfactuals}, who describes it as ``truth-functional logic in disguise''. However, as the foregoing results indicate, truth-functional logic can come in several disguises.\footnote{In this sense, the diagram of variably strict conditional logics given in \cite[p.~131]{Lewis1973Counterfactuals} is notably incomplete. However, if we restrict our attention to variably strict conditional logics satisfying $(\varphi\boxright\psi)\to(\varphi\to\psi)$, then it is clear that \textsf{VCA} is the only Post complete conditional logic in that class.}

    The following corollary of Theorem~\ref{Theorem:RegularPCLogics} may have some interest to specialists. Let us call a conditional logic \textit{connexive} if it contains the following schema:\footnote{For recent discussion of the historical origins of connexivity, consult Lenzen~\cite{Lenzen2020critical}, Weiss~\cite{Weiss2022Aristotle}, and Ruge~\cite{Ruge2023Connexivity}.}
    \begin{equation*}
        \label{Axiom:Aristotle}
        \tag{Aristotle}
        \neg(\neg\varphi\boxright\varphi)
    \end{equation*}
    It is easy to see that there are no consistent normal connexive conditional logics (cf.~Weiss~\cite[p.~614]{Weiss2019ConnexiveExtensions}), and that the only Post complete regular connexive conditional logics are $\mathsf{CR}\oplus{\text{\ref{Axiom:Bot}}}$ and $\mathsf{CR}\oplus{\text{\ref{Axiom:Conjunction}}}$. These two systems comprise the sole Post complete regular extensions of the connexive conditional logics studied by Weiss~\cite{Weiss2019ConnexiveExtensions}.

    Before continuing, let us observe that every Post complete normal conditional logic admits a proposition selection frame.
    
    \begin{lemma}
        {\label{lem:postcompletenormalhaveframes}}
        Every Post complete normal conditional logic has a proposition selection frame (i.e., there is a frame $\mathfrak{F}^\pi$ such that $\mathsf{L} \subseteq \mathsf{L}(\mathfrak{F}^\pi)$).
    \end{lemma}
    \begin{proof}
        We present the frames for the systems enumerated in Corollary~\ref{Corollary:NormalPCLogics} and leave the routine verification of soundness to the reader.
        \begin{itemize}
        
            \item[$(i)$] The frame for $\mathsf{CR} \oplus \textsf{\ref{Axiom:Top}}$ is  $\mathfrak{F}^{\pi}_{\text{\ref{Axiom:Top}}}=\langle{\{a\},\pi}\rangle$, where for $X\subseteq\{a\}$: $$\pi(a, X)=\emptyset$$

            \item[$(ii)$] The frame for $\mathsf{CR} \oplus \textsf{\ref{Axiom:Disjunction}}$ is $\mathfrak{F}^{\pi}_{\text{\ref{Axiom:Disjunction}}}=\langle{\{a\},\pi}\rangle$, where for $X\subseteq\{a\}$: $$\pi(a, X)=\begin{cases}\emptyset & \hfill \text{ if }a \in X\\ \{a\} & \hfill \text{else}\end{cases}$$ 
           
            \item[$(iii)$] The frame for $\mathsf{CR} \oplus \textsf{\ref{Axiom:Imp}}$ is $\mathfrak{F}^{\pi}_{\text{\ref{Axiom:Imp}}}=\langle{\{a\},\pi}\rangle$, where for $X\subseteq\{a\}$: $$\pi(a, X)=\begin{cases}\{a\} & \hfill \text{ if }a \in X\\ \emptyset & \hfill \text{else}\end{cases}$$ 

            \item[$(iv)$] The frame for $\mathsf{CR} \oplus \textsf{\ref{Axiom:Consequent}}$ is  $\mathfrak{F}^{\pi}_{\text{\ref{Axiom:Consequent}}}=\langle{\{a\},\pi}\rangle$, where for $X\subseteq\{a\}$: $$\pi(a, X)=\{a\} $$ 
            
        \end{itemize}
    \end{proof}

    We now establish analogues for modalized and variably strict conditional logic of some embedding results for modal logic due to Makinson~\cite{Makinson1971} (cf.~Segerberg~\cite[p.~712]{Segerberg1972PostCompleteness}).

    \begin{theorem}
        \label{Theorem:CoatomsNExtP}
        Every consistent conditional logic in $\NExt(\mathsf{M})$ is contained in at least one of $\mathsf{CR}\oplus{\text{\em \ref{Axiom:Top}}}$  and  $\mathsf{CR}\oplus{\text{\em \ref{Axiom:Imp}}}$.
    \end{theorem}
    \begin{proof}
        We show that for any consistent conditional logic $\mathsf{L} \in \NExt(\mathsf{M})$, if $\mathsf{L} \nsubseteq \mathsf{CR}\oplus{\text{\ref{Axiom:Top}}}$, then  $\mathsf{L} \subseteq \mathsf{CR}\oplus{\text{\ref{Axiom:Imp}}}$, which is equivalent to the claim to be proved.

        Consider any $\mathsf{L}\in \NExt(\mathsf{M})$ and assume $\mathsf{L} \nsubseteq \mathsf{CR}\oplus{\text{\ref{Axiom:Top}}}$. Now, consider the following set:
        \[
            \mathsf{L}'=\{\varphi\text{ :}\vdash_{\mathsf{L}}(\top\boxright\bot)\to\varphi\}
        \]
        It is straightforward to observe that $\mathsf{L} \subseteq \mathsf{L}'$. Now, we verify that $\mathsf{L}'$ is indeed a regular logic. Since $\mathsf{L} \subseteq \mathsf{L}'$, $\mathsf{L}'$ contains all classical propositional tautologies, \ref{Axiom:CM}, \ref{Axiom:CC}, and \ref{Axiom:MOD}. Moreover, if $\varphi, \varphi \to \psi \in \mathsf{L}'$, then: 
        \begin{flalign*}
& (i) && \vdash_\mathsf{L} (\top \boxright \bot) \to \varphi &&  \\
& (ii) && \vdash_\mathsf{L} (\top \boxright \bot) \to (\varphi \to \psi) && \\
& (iii) && \vdash_\mathsf{L} (\top \boxright \bot) \to \psi && \text{by }(i)\text{ and }(ii)\\
& (iv) && \psi \in \mathsf{L}' && \text{by } (iii)
        \end{flalign*}
        
    So $\mathsf{L}'$ is closed under modus ponens. It is also closed under uniform substitution since $\mathsf{L}$ is: for any uniform substitution $^\sigma$, if $\varphi \in \mathsf{L}'$, then  $\vdash_{\mathsf{L}}(\top\boxright\bot)\to\varphi$; thus $\vdash_{\mathsf{L}}((\top\boxright\bot)\to\varphi)^\sigma$; but $((\top\boxright\bot)\to\varphi)^\sigma=(\top\boxright\bot)\to\varphi^\sigma$ therefore $\varphi^\sigma \in \mathsf{L}'$. We now show that every formula of the form $\varphi \boxright \psi$ belongs to $\mathsf{L}'$. In fact, for arbitrary formulas $\varphi,\psi$, we have:
        \begin{flalign*}
& (i) && \vdash_\mathsf{L} (\top\boxright\bot)\to(\varphi\boxright\bot) && \text{by \ref{Axiom:MOD}} \\
& (ii) && \vdash_\mathsf{L} (\varphi\boxright\bot)\leftrightarrow(\varphi\boxright(\bot\wedge\psi)) && \text{by \ref{Rule:RCEC}} \\
& (iii) && \vdash_\mathsf{L} (\varphi\boxright(\bot\wedge\psi))\to(\varphi\boxright\psi) && \text{by \ref{Axiom:CM}} \\
& (iv) && \vdash_\mathsf{L} (\varphi\boxright\bot)\to(\varphi\boxright\psi) && \text{by }(ii)\text{ and }(iii) \\
& (v) && \vdash_\mathsf{L} (\top\boxright\bot)\to(\varphi\boxright\psi) && \text{by }(i)\text{ and }(iv) \\
& (vi) && \varphi\boxright\psi\in\mathsf{L}' && \text{by }(v)
        \end{flalign*}

        Now, suppose that $\varphi\leftrightarrow\psi \in \mathsf{L}'$; the preceding result gives that $\delta \boxright \psi \in \mathsf{L}'$ and $\delta \boxright \varphi \in \mathsf{L}'$. Hence, we have:  
         \begin{flalign*}
        & (i)    && \vdash_\mathsf{L} (\top\boxright\bot)\to(\varphi\boxright\delta) && \\
            & (ii)    && \vdash_\mathsf{L} (\top\boxright\bot)\to(\psi\boxright\delta) && \\
           & (iii) && \vdash_{\mathsf{L}}(\top\boxright\bot)\to
            ((\varphi\boxright\delta)\leftrightarrow(\psi\boxright\delta)) &&\text{by }(i)\text{ and }(ii)\\
            & (iv) && (\varphi\boxright\delta)\leftrightarrow(\psi\boxright\delta) \in \mathsf{L}' && \text{by }(iii)
        \end{flalign*}
        Therefore $\mathsf{L}'$ is closed under \ref{Rule:RCEA}. By analogous reasoning, we can show that $\mathsf{L}'$ is closed under \ref{Rule:RCEC} too. So, $\mathsf{L}'$ is a regular conditional logic. 
        We also verify that $\mathsf{L}'$ contains \ref{Axiom:Top}. For arbitrary formulas $\varphi,\psi$, we have: \begin{flalign*}
        & (i) && \varphi\boxright\psi\in\mathsf{L}' && \\
        & (ii) &&(\varphi\boxright\psi)\to\bigl((\varphi\boxright\psi)\leftrightarrow\top\bigr)\in\mathsf{L}' && \text{by }(i) \\
        & (iii) && (\varphi\boxright\psi)\leftrightarrow\top\in\mathsf{L}' && \text{by }(i), (ii)\text{ and modus ponens}
        \end{flalign*}
        Therefore, since $\mathsf{L}'$ is a regular conditional logic and contains \ref{Axiom:Top}, we have that $\mathsf{CR}\oplus\textsf{\ref{Axiom:Top}}\subseteq\mathsf{L}'$. However, since $\mathsf{L} \nsubseteq \mathsf{CR}\oplus\textsf{\ref{Axiom:Top}}$ by assumption, there is some $\varphi \in \mathsf{L}$ such that $\varphi \notin \mathsf{CR}\oplus\textsf{\ref{Axiom:Top}}$. Moreover, $\mathsf{L} \subseteq \mathsf{L}'$, therefore $\varphi \in \mathsf{L}'$, hence $\mathsf{CR}\oplus\textsf{\ref{Axiom:Top}}\subset\mathsf{L}'$. However, since $\mathsf{CR}\oplus\textsf{\ref{Axiom:Top}}$ is Post complete, it must be the case that $\mathsf{L}'=Fm$, which implies that $\bot \in \mathsf{L}'$; thus, $\vdash_\mathsf{L} (\top \boxright \bot) \to \bot$, namely $\vdash_\mathsf{L} \neg(\top\boxright \bot)$. 
        
        Moreover, since $\mathsf{L} \in \NExt(\mathsf{M})$, we have that $\vdash_\mathsf{L}\top\boxright\top$ and $\vdash_\mathsf{L}\bot\boxright\bot$ (by \ref{Axiom:ID}) and $\vdash_\mathsf{L}\bot\boxright\top$ (by \ref{Axiom:CN}). That is,
        \[
            \text{(\ref{RegularSet:IMP})}
            =\{\top\boxright\top,\neg(\top\boxright\bot),
                \bot\boxright\top,\bot\boxright\bot\}
            \subseteq\mathsf{L},
        \]
        hence,  $\mathsf{CR}\oplus\text{(\ref{RegularSet:IMP})}\subseteq\mathsf{L}$.
        
        By Lemmata~\ref{Lemma:SegerbergA} and \ref{Lemma:RegularSets}, both $\mathsf{L}$ and  $\mathsf{CR}\oplus\text{(\ref{RegularSet:IMP})}$ have a unique Post complete extension. Clearly, $\mathsf{CR}\oplus\text{(\ref{RegularSet:IMP})}\subseteq \mathsf{CR}\oplus\text{\ref{Axiom:Imp}}$ and $\mathsf{CR}\oplus\text{\ref{Axiom:Imp}}$ is the unique Post complete extension of $\mathsf{CR}\oplus\text{(\ref{RegularSet:IMP})}$. Now, if $\mathsf{L}$ had another Post complete extension $\mathsf{L}^*$ different from $\mathsf{CR}\oplus\text{\ref{Axiom:Imp}}$, this would contradict the fact that  $\mathsf{CR}\oplus\text{(\ref{RegularSet:IMP})}$ has a unique post complete extension, since $\mathsf{CR}\oplus\text{(\ref{RegularSet:IMP})} \subseteq \mathsf{L}$. Therefore, $\mathsf{CR}\oplus\text{\ref{Axiom:Imp}}$ is the only Post complete extension of $\mathsf{L}$, namely $\mathsf{L} \subseteq \mathsf{CR}\oplus\text{\ref{Axiom:Imp}}$, completing the proof.
    \end{proof}

    As a corollary of Theorem~\ref{Theorem:CoatomsNExtP} and the fact that $\mathsf{M}\subseteq\mathsf{V}$, we obtain the following.

    \begin{corollary}
        \label{Corollary:VariablyStrictCoatoms}
        Every consistent conditional logic in $\NExt(\mathsf{V})$ is contained in at least one of $\mathsf{CR}\oplus{\text{\em \ref{Axiom:Top}}}$  and $\mathsf{CR}\oplus{\text{\em \ref{Axiom:Imp}}}$.
    \end{corollary}

    From the preceding and Lemma \ref{lem:postcompletenormalhaveframes}, we may easily infer:

    \begin{corollary}
    \label{Corollary:EveryNExtPhasaframe}
        Every consistent logic $\mathsf{L} \in \NExt(\mathsf{M})$ has a proposition selection frame.
    \end{corollary}
\begin{proof}
    The relevant frames are exactly those defined in points $(i)$ and $(iii)$ in the proof of Lemma~\ref{lem:postcompletenormalhaveframes}.
\end{proof}

   % \begin{proof}
   % By Theorem \ref{Theorem:CoatomsNExtP}, it suffices to prove that each of $\mathsf{CR} \oplus \text{\ref{Axiom:Top}}$ and $\mathsf{CR} \oplus \text{\ref{Axiom:Imp}}$ has a proposition selection frame.  We present the frames and leave the routine verification of soundness to the reader.

   % \begin{itemize}
        
    %        \item[$i.$] The frame for $\mathsf{CR} \oplus \textsf{\ref{Axiom:Top}}$ is  $\mathfrak{F}^{\pi}_{\text{\ref{Axiom:Top}}}=\langle{\{a\},\pi}\rangle$, where for $X\subseteq\{a\}$: $$\pi(a, X)=\emptyset$$
           
     %       \item[$ii.$] The frame for $\mathsf{CR} \oplus \textsf{\ref{Axiom:Imp}}$ is $\mathfrak{F}^{\pi}_{\text{\ref{Axiom:Imp}}}=\langle{\{a\},\pi}\rangle$, where for $X\subseteq\{a\}$: $$\pi(a, X)=\begin{cases}\{a\} & \hfill \text{ if }a \in X\\ \emptyset & \hfill \text{else}\end{cases}$$ 
            
      %  \end{itemize}
    %\end{proof}

    As a consequence of the above result and Corollary \ref{Corollary:VariablyStrictCoatoms},  we get the following.

    \begin{corollary}
        \label{Corollary:EveryVariablyStricthasaframe}
        Every consistent logic $\mathsf{L} \in \NExt(\mathsf{V})$ has a proposition selection frame.
    \end{corollary}

    %{\color{blue} It would be useful maybe to include some discussion on the role of ID and MOD. They play the special role of making maximal normal logics containing ID and MOD to be post complete as well. So, somehow, the lattice $\NExt(\mathsf{M})$ is well behaved in this sense.} {\color{red} Yes, they play a significant role in the `conditional companion' stuff too. Perhaps we can ``tease'' some of our more recent/in-progress results here.}

Every consistent normal \textit{monomodal} logic is contained in one of two Post complete normal modal logics, \textsf{TR} and \textsf{VER}, and each has a frame.\footnote{See Makinson~\cite[p.~253]{Makinson1971} and Segerberg~\cite[p.~712]{Segerberg1972PostCompleteness}. The names (approximately) follow Hughes and Cresswell~\cite[pp.~64--68]{HughesCresswell1996}.} So, one might naturally have conjectured that Theorem~\ref{Theorem:CoatomsNExtP} can be strengthened to show that every consistent \textit{normal} conditional logic is contained in one of the four Post complete normal conditional logics listed in Corollary~\ref{Corollary:NormalPCLogics}. Similarly, one might have conjectured that Corollary~\ref{Corollary:EveryNExtPhasaframe} can be strengthened to show that every consistent normal conditional logic has a frame. We will presently demonstrate that both conjectures are false. By adapting a result of Thomason~\cite[\S4]{Thomason1972} (cf.~Blackburn et al.~\cite[\S4.4]{Blackburn2001}) concerning bimodal tense logic, we will display an incomplete normal conditional logic that refutes both conjectures.\footnote{For other incompleteness results in conditional logic, see Nute~\cite{Nute1978Incompleteness} and Kocurek et al.~\cite{KocurekWalshWeiss}. Incidentally, the result reported in Nute~\cite{Nute1978Incompleteness} bootstraps on a different incompleteness result due to Thomason~\cite{Thomason1974Incompleteness}.}

\begin{definition}
    %\label{def:tenselogic}
    Let $\mathsf{CT}$ be the smallest normal conditional logic containing the following axioms (i.e., $\mathsf{CT} = \mathsf{CK} \oplus \{\text{\ref{Axiom:T1}}, \text{\ref{Axiom:T2}}, \text{\ref{Axiom:linear}},\text{\ref{Axiom:Seriality}}, \text{\ref{Axiom:LOB}}\}$):
    \begin{equation*}
        \label{Axiom:T1}
        \tag{T1}
        \varphi \to \square_\top\lozenge_\bot \varphi
    \end{equation*}
    \begin{equation*}
        \label{Axiom:T2}
        \tag{T2}
        \varphi \to \square_\bot\lozenge_\top \varphi
    \end{equation*}
    \begin{equation*}
        \label{Axiom:linear}
        \tag{LIN}
        (\lozenge_\top \varphi \wedge \lozenge_\top\psi)\to (\lozenge_\top(\varphi \wedge \psi) \vee \lozenge_\top(\varphi \wedge \lozenge_\top\psi)\vee \lozenge_\top(\psi \wedge\lozenge_\top\varphi))
    \end{equation*}
    \begin{equation*}
        \label{Axiom:Seriality}
        \tag{CD}
        \square_\top \varphi \to \lozenge_\top \varphi
    \end{equation*}
    \begin{equation*}
        \label{Axiom:LOB}
        \tag{L{\"O}B}
        \square_{\bot}(\square_\bot\varphi\to\varphi)\to\square_\bot\varphi
    \end{equation*}
\end{definition}

\noindent
The reader may recognize \ref{Axiom:LOB} as a form of Löb's axiom from provability logic; note that $\square_\bot\varphi\to\square_\bot\square_\bot\varphi$ follows from this (cf.~Boolos~\cite[p.~11]{Boolos1993}). We now introduce an important family of proposition selection frames.

\begin{definition}
    %\label{def:tenseframe}
    A \emph{tense frame} is a proposition selection frame $\mathfrak{F}^\pi = \langle W, \pi \rangle$ satisfying the following properties (define $x <_\top y$ iff $y \in \pi(x, W)$; $x <_\bot y$ iff $y \in \pi(x, \emptyset)$):
    \begin{equation*}
        \label{Prop:T1}
        \tag{Tf}
        x <_\top y\text{ iff }y <_\bot x
    \end{equation*}
    \begin{equation*}
        \label{Prop:C4}
        \tag{4f}
        \text{if ($x<_\top y$ and $y<_\top z$) then $x<_\top z$}
    \end{equation*}
    \begin{equation*}
        \label{Prop:linear}
        \tag{LINf}
        \begin{split}
        \text{if }(x<_\top y\text{ and }x<_\top z)\text{ then }(y <_\top z\text{ or }z <_\top y\text{ or }y=z)
        \end{split}
    \end{equation*}
    \begin{equation*}
        \label{Prop:Seriality}
        \tag{CDf}
        \text{for all }w \in W,\text{ there is some }u \in W\text{ such that }w<_\top u
    \end{equation*}
    \begin{equation*}
        \label{Prop:LOB}
        \tag{L{\"O}Bf}
        \begin{split}
        \text{for all }\emptyset\neq S\subseteq W,\text{ there is a }w \in S\text{ such that for all }u \in S, w\not<_\bot u
        \end{split}
    \end{equation*}
\end{definition}

The strategy of the argument is as follows. By showing that (1) only tense frames validate \textsf{CT} and (2) no tense frame validates:
    \begin{equation*}
        \label{Axiom:McK}
        \tag{McK}
        \square_\top\lozenge_\top\varphi \to \lozenge_\top\square_\top\varphi
    \end{equation*}
we thereby establish that $\mathsf{CT} \oplus \text{\ref{Axiom:McK}}$ has no frames. However, we also show that $\mathsf{CT} \oplus \text{\ref{Axiom:McK}}$ is consistent, from which incompleteness and other results follow.

\begin{lemma}
    \label{Lemma:CtCorrespondence}
    For any proposition selection frame $\mathfrak{F}^\pi=\langle W, \pi \rangle$, $\mathsf{CT} \subseteq \mathsf{L}(\mathfrak{F}^\pi)$ only if $\mathfrak{F}^\pi$ is a tense frame.
\end{lemma}
\begin{proof}
    We focus on a single case and leave the rest to the reader. Specifically, we demonstrate that, for any proposition selection frame $\mathfrak{F}^\pi=\langle W, \pi\rangle$, if $\models^{\mathfrak{F}^\pi} \text{\ref{Axiom:T1}}$ and $\models^{\mathfrak{F}^\pi} \text{\ref{Axiom:T2}}$, then \ref{Prop:T1} holds for $\mathfrak{F}^\pi$. Suppose \ref{Prop:T1} fails; then, for some $x, y \in W$, either $(i)$ $x <_\top y$ and $y \not<_\bot x$ or $(ii)$ $x \not<_\top y$ and $y <_\bot x$. If $(i)$ holds, consider any model $\mathcal{M}^\pi=\langle W, \pi, V\rangle$ over $\mathfrak{F}^\pi$ such that $V(p)=\{x\}$. Since $y \not<_\bot x$, we have that $\not\models^{\mathcal{M}^\pi}_y \lozenge_\bot p$; but as $x <_\top y$, $\not\models^{\mathcal{M}^\pi}_x \square_\top \lozenge_\bot p$. Therefore, $\not\models^{\mathfrak{F}^\pi} \text{\ref{Axiom:T1}}$. If $(ii)$ holds, a symmetric argument shows that $\not\models^{\mathfrak{F}^\pi} \text{\ref{Axiom:T2}}$. The reader may verify the following implications: $\models^{\mathfrak{F}^\pi} \text{\ref{Axiom:linear}}$ implies \ref{Prop:linear}; $\models^{\mathfrak{F}^\pi} \text{\ref{Axiom:Seriality}}$ implies \ref{Prop:Seriality}; and, under the assumption that \ref{Prop:T1} holds for $\mathfrak{F}^\pi$, $\models^{\mathfrak{F}^\pi} \text{\ref{Axiom:LOB}}$ implies both \ref{Prop:LOB} and \ref{Prop:C4}.
\end{proof}

Where $\langle{S,<}\rangle$ is a strict linear order, we call $T\subseteq S$ \textit{cofinal} in $S$ if $\forall{s\in{S}},\exists{t\in{T}}(s<t)$.

\begin{lemma}
    \label{Lemma:NoMcK}
    For every tense frame $\mathfrak{F}^\pi$, $\not\models^{\mathfrak{F}^\pi} \square_\top\lozenge_\top p \to \lozenge_\top\square_\top p$.
\end{lemma}
\begin{proof}
    Consider a tense frame $\mathfrak{F}^\pi=\langle W, \pi \rangle$. $W\neq\emptyset$, so pick a $w\in W$ and consider $S := \{x : w <_\top x\}$; clearly, this is a nonempty strict linear order with no greatest element. By the axiom of choice, we can define $\emptyset\neq T\subset S$ such that both $T$ and $S\setminus{T}$ are cofinal in $S$. 
    
    Now consider any model $\mathfrak{M}^\pi=\langle W, \pi, V \rangle$ over $\mathfrak{F}^\pi$ such that $V(p)=T$. For any $x$ such that $w <_\top x$, by the cofinality of $T$ and $S\setminus{T}$, there are $w_1 \in T$ and $w_2 \in S\setminus{T}$ such that $(i)$ $x <_\top w_1$ and $(ii)$ $x <_\top w_2$. Thus, by $(i)$, $\models^{\mathfrak{M}^\pi}_x \lozenge_\top p$, and by $(ii)$, $\not\models^{\mathfrak{M}^\pi}_x \square_\top p$, whence $\models^{\mathfrak{M}^\pi}_w \square_\top\lozenge_\top p$ but $\not\models^{\mathfrak{M}^\pi}_w \lozenge_\top\square_\top p$. Therefore, $\not\models^{\mathfrak{F}^\pi} \text{\ref{Axiom:McK}}$.
\end{proof}

\begin{theorem}
    \label{Theorem:NoFramesforCt}
    There is no proposition selection frame $\mathfrak{F}^\pi$ such that $\mathsf{CT} \oplus \text{\em \ref{Axiom:McK}} \subseteq \mathsf{L}(\mathfrak{F}^\pi)$.
\end{theorem}
\begin{proof}
    For reductio, suppose that there were a proposition selection frame $\mathfrak{F}^\pi$ such that $\mathsf{CT} \oplus \text{\ref{Axiom:McK}} \subseteq \mathsf{L}(\mathfrak{F}^\pi)$. By Lemma~\ref{Lemma:CtCorrespondence}, $\mathfrak{F}^\pi$ must be a tense frame. However, by Lemma~\ref{Lemma:NoMcK}, $\not\models^{\mathfrak{F}^\pi} \text{\ref{Axiom:McK}}$, which contradicts our assumption.
\end{proof}

Thomason's incomplete consistent tense logic, \textsf{Thom}, can be formulated in a bimodal language $\mathcal{L}^\textsf{Thom}$ differing from $\mathcal{L}$ only in substituting the primitive unary modalities $\mathcal{G}$ (`always \textit{going} to be that') and $\mathcal{H}$ (`always \textit{has} been that') for $\boxright$. Let $Fm^{\textsf{Thom}}$ be the set of all formulas in $\mathcal{L}^\textsf{Thom}$. The axioms of \textsf{Thom} are essentially those of $\mathsf{CT} \oplus \text{\ref{Axiom:McK}}$, systematically replacing $\Box_\top$ by $\mathcal{G}$ and $\Box_\bot$ by $\mathcal{H}$.

\begin{lemma}
    \label{Lemma:CtMcKinseyConsistent}
    $\mathsf{CT} \oplus \text{\em \ref{Axiom:McK}}$ is consistent.
\end{lemma}
\begin{proof}
    Consider the translation $^\mu:Fm\to Fm^\textsf{Thom}$ defined as follows:
    \begin{itemize}
        \item[$p$.] $p^\mu = p$;
        \item[$\bot$.] $\bot^\mu = \bot$;
        \item[$\neg$.] $(\neg\varphi)^\mu = \neg\varphi^\mu$;
        \item[$\vee$.] $(\varphi\vee\psi)^\mu = \varphi^\mu \vee \psi^\mu$;
        \item[$\boxright$.] $(\varphi\boxright\psi)^\mu = \begin{cases}
            \mathcal{G}\psi^\mu & \text{if }\vdash_\mathsf{Thom} \varphi^\mu \\
            \mathcal{H}\psi^\mu & \text{else}
        \end{cases}$
    \end{itemize}
    Observe that the translation $^\mu$ ensures that $\bot^\mu=\bot$ and $\top^\mu = \top$. Therefore, $(\top\boxright\varphi)^\mu = \mathcal{G}\varphi^\mu$, and by Thomason's proof that  $\not\vdash_{\textsf{Thom}}\bot$, $(\bot\boxright\varphi)^\mu = \mathcal{H}\varphi^\mu$. 
    
    It must be verified that if $\vdash_{\mathsf{CT} \oplus \text{\ref{Axiom:McK}}}\varphi$, then $\vdash_{\textsf{Thom}}\varphi^\mu$. This is proved by induction on the length of proof; we briefly examine the case of \ref{Rule:RCEA} and leave the remaining cases to the reader. Suppose it has been shown that $\vdash_{\mathsf{Thom}}\varphi^\mu\leftrightarrow\psi^\mu$. Assume that $\vdash_{\mathsf{Thom}}\varphi^\mu$; then $\vdash_{\mathsf{Thom}}\psi^\mu$, so $\vdash_{\mathsf{Thom}}\mathcal{G}\theta^\mu\leftrightarrow\mathcal{G}\theta^\mu = ((\varphi\boxright\theta)\leftrightarrow(\psi\boxright\theta))^\mu$ as desired. Now assume instead that $\not\vdash_{\mathsf{Thom}}\varphi^\mu$; then $\not\vdash_{\mathsf{Thom}}\psi^\mu$, so $\vdash_{\mathsf{Thom}}\mathcal{H}\theta^\mu\leftrightarrow\mathcal{H}\theta^\mu = ((\varphi\boxright\theta)\leftrightarrow(\psi\boxright\theta))^\mu$ as desired. The principal result follows from the foregoing and Thomason's proof that \textsf{Thom} is consistent (i.e., $\not\vdash_{\textsf{Thom}}\bot$). 
\end{proof}

\begin{corollary}
    \label{Corollary:FrameFailure}
    There are consistent systems in $\NExt(\mathsf{CK})$ that do not have proposition selection frames.
\end{corollary}
\begin{proof}
    Clearly, $\mathsf{CT} \oplus \text{\ref{Axiom:McK}} \in \NExt(\mathsf{CK})$ and it is consistent by Lemma~\ref{Lemma:CtMcKinseyConsistent}. The result follows from Theorem~\ref{Theorem:NoFramesforCt}.
\end{proof}

\begin{corollary}
    \label{Corollary:EmbeddingFailure}
    There are consistent systems in $\NExt(\mathsf{CK})$ not contained in any Post complete system in $\NExt(\mathsf{CK})$.
\end{corollary}
\begin{proof}
    Since all of the normal Post complete conditional logics have frames by Lemma~\ref{lem:postcompletenormalhaveframes}, if $\mathsf{CT} \oplus \text{\ref{Axiom:McK}}$ were contained in any of them, it would have a proposition selection frame contradicting Theorem~\ref{Theorem:NoFramesforCt}.
\end{proof}

A normal conditional logic $\mathsf{L}$ is \textit{(proposition selection) frame incomplete} if for no class of (proposition selection) frames $\mathcal{C}$, $\mathsf{L} = \mathsf{L}(\mathcal{C})$.

\begin{corollary}
    %\label{Corollary:SeminormalIncompleteness}
    $\mathsf{CT} \oplus \text{\em \ref{Axiom:McK}}$ is frame incomplete.
\end{corollary}
\begin{proof}
    Suppose there were a class of proposition selection frames $\mathcal{C}$ such that $\mathsf{CT} \oplus \text{\ref{Axiom:McK}} = \mathsf{L}(\mathcal{C})$. By Theorem~\ref{Theorem:NoFramesforCt}, since $\mathsf{CT} \oplus \text{\ref{Axiom:McK}} \subseteq \mathsf{L}(\mathcal{C})$, it must be that $\mathcal{C} = \emptyset$. Then $\bot\in\mathsf{L}(\mathcal{C})$, whence $\bot\in\mathsf{CT} \oplus \text{\ref{Axiom:McK}}$, contradicting Lemma~\ref{Lemma:CtMcKinseyConsistent}.
\end{proof}

    \subsection{Further results on Post completeness}
    \label{Subsection:Further Post complete extensions}

    Segerberg~\cite[pp.~713ff.]{Segerberg1972PostCompleteness} shows that several modal logics, for example \textsf{K}, have $2^{\aleph_{0}}$ many Post complete extensions. It is clear that \textit{essentially all} of these Post complete extensions are not particularly well-behaved logics. For example, \textsf{K} has only two normal Post complete extensions, so essentially all (i.e., all but a small finite number) of its Post complete extensions are non-normal. 
    
    In this section, we will see that the situation in conditional logic is both similar and different. We first show that \textsf{CK} (and notable subsystems) likewise has $2^{\aleph_{0}}$ many Post complete extensions, essentially all of which are non-normal. However, we then show that the seminormal conditional logic \textsf{Ck} has $2^{\aleph_{0}}$ many \textit{seminormal} Post complete extensions. 

    To show that \textsf{CK} has $2^{\aleph_{0}}$ many Post complete extensions, we show that $\NExt(\mathsf{CK})$ has $2^{\aleph_{0}}$ many coatoms. We obtain as corollaries counts of coatoms in various other lattices of conditional logics as well as counts of Post complete extensions of various other systems.

    \begin{definition}
     \label{def:broom}
        For $I \subset \mathbb{N}^+$, the \textit{$I$-nest frame} is the proposition selection frame $\mathfrak{F}^{I} = \langle \Omega^I, \pi^I \rangle$ where:
    \begin{itemize}
        \item $\Omega^{I} = \{0\} \cup \{ (n, k) : n \in \mathbb{N}^+\setminus I, 1 \leq k \leq n \}$;

        \item For $X \subseteq \Omega^I$:

        \begin{itemize}
            \item[i.] $\pi^I(0, X)=\begin{cases}\{(n, m) \in \Omega^I\setminus\{0\} : m=1 \} & \hfill \text{if }X\neq \emptyset \\ \{0\} & \hfill \text{else.}\end{cases}$

            \item[ii.] $\pi^I((i, j), X)=\begin{cases} \{(i, j+1)\} & \hfill \text{if } j < i\text{ and }X \neq \emptyset\\ \{0\} & \hfill X = \emptyset\\ \emptyset & \hfil \text{else.} \end{cases}$
        \end{itemize}
    \end{itemize}
    \end{definition}

With respect to the $I$-nest frame $\mathfrak{F}^{I} = \langle \Omega^I, \pi^I \rangle$, we define $R^I \subseteq \Omega^I \times \Omega^I$ as follows:
\begin{itemize}
    \item $w R^I y \Leftrightarrow y \in \pi^I(w, X)$ for some non-empty $X \subseteq \Omega^{I}$.
\end{itemize}

\noindent Moreover, for $n \in \mathbb{N}$, $R^I_n \subseteq \Omega^I \times \Omega^I$ is inductively defined as follows:

\begin{itemize}
\item $x R^I_0 y \Leftrightarrow x = y$; $x R^I_{n+1} y \Leftrightarrow \text{there is a $z \in \Omega^I$ such that }xR^I_n z\text{ and }zR^I y$.
%\item $x R^I_{n+1} y \Leftrightarrow \text{ there is a $z \in \Omega^I$ such that }wR^I_n z\text{ and }zR^I y$
\end{itemize}

\noindent Finally, using these relations, we define $R^I[w]=\{x \in \Omega^I : w R^I x\}$ and $R^I_n[w]=\{x \in \Omega^I : w R^I_n x\}$.

Essentially the $I$-nest frame can be interpreted as a Kripke frame, that is, a degenerate proposition selection frame in which the accessibility relation $R^I$ interprets the proposition selection function $\pi^I$ when its second coordinate is non-empty. In each pair $(i, j)$, the first coordinate $i$ specifies the chain, while the second coordinate $j$ specifies the position within that chain. For each state $(i, j)$, the selection function $\pi^{I}((i, j), X)$ simply selects the successor of $(i, j)$ in the $i$-th chain up to $(i,i)$, when $X\neq \emptyset$, and $0$ otherwise. If $I = \emptyset$, the frame consists roughly of all finite chains of natural numbers rooted at $0$, as illustrated in Figure~\ref{Figure:EmptyNestFrame}, where the arrows represent the relation $R^I$ and, additionally, each point $w \in \Omega^I$ \emph{accesses} $0$ via the different relation induced by $\pi(w, \emptyset)$ (this last relation is not displayed in the figure).

\begin{figure}[h]
\centering
\begin{tikzcd}
         &                                                                                                                        &                     &                     & {(3, 3)} \\
         & {(2, 2)}                                                                                                               &                     & {(3, 2)} \arrow[ru] &          \\
{(1, 1)} & {(2, 1)} \arrow[u]                                                                                                     & {(3, 1)} \arrow[ru] & {}                  &          \\
         & 0 \arrow[lu] \arrow[u] \arrow[ru] \arrow[rru, no head, dashed] \arrow[rr, no head, dashed] \arrow[rd, no head, dashed] &                     & {}                  &          \\
         &                                                                                                                        & {}                  &                     &         
\end{tikzcd}
\caption{The $\emptyset$-nest frame.}
\label{Figure:EmptyNestFrame}
\end{figure}
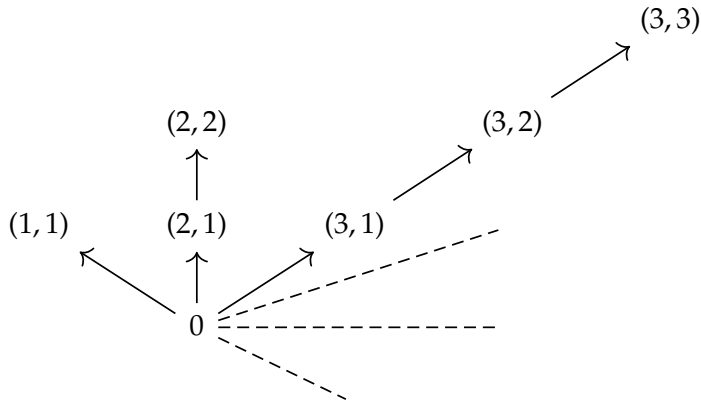

Where $\emptyset \neq I\subset\mathbb{N}^{+}$, the corresponding $I$-nest frame can roughly be regarded as derived from the $\emptyset$-nest frame by removing all branches indexed by $I$; that is, by excluding all pairs whose first coordinate is an element of $I$.

\begin{lemma}
\label{lem:propofbroomframe}
For any $I$-nest frame $\mathfrak{F}^{I} = \langle \Omega^I, \pi^I \rangle$:
    \begin{itemize}

        \item[i.] For all $w \in \Omega^I$ and all non-empty $X, Y \subseteq \Omega^I$, $\pi^I(w, X)=\pi^I(w, Y)=R^I[w]$;

        \item[ii.] For all $w \in \Omega^I$, $\pi^I(w, \emptyset)=\{0\}$;

        \item[iii.] For all $(i, j) \in \Omega^I \setminus \{0\}$, $R^I[(i, j)]=\begin{cases} \{(i, j+1)\} & \hfill \text{ if }j < i \\ \emptyset & \text{ else;} \end{cases}$

        \item[iv.] For all $n \in \mathbb{N}^+$, $R^I_n[0]=\{(i, j) : (i, j) \in \Omega^I, j=n\}$;
        \item[v.] For all $n \in \mathbb{N}^+\setminus I$, $R^I_n[0]\neq \emptyset$;
        \item[vi.] For all $n \in \mathbb{N}^+$, $(n, n) \in \Omega^I \Leftrightarrow n \notin I$.

    \end{itemize}
\end{lemma}
\begin{proof} 
Points (i), (ii) and (iii) are straightforward by construction and the definition of $R^I$. Point (iv) is easily proved by induction on $n$ using point (iii) and the definition of $R^I_n$. Points (v) and (vi) follow from Definition \ref{def:broom}.
\end{proof}

Recall the inductive definitions of $\square^{n}_\varphi$ and $\lozenge^{n}_\varphi$ from Section~\ref{Section:Logics}.

\begin{lemma}
    \label{Lemma:IteratedTruthConditions}
    For any $I$-nest model $\mathfrak{M}^I=\langle \Omega^I, \pi^I, V \rangle$ and $w\in \Omega^I$:
    \begin{itemize}
        \item[i.] $\models^{\mathfrak{M}^I}_w \square^{n}_{\top}\varphi \Leftrightarrow R^{I}_{n}[w]\subseteq[\varphi]^{\mathfrak{M}^I}$;
        \item[ii.] $\models^{\mathfrak{M}^I}_w \lozenge^{n}_{\top}\varphi \Leftrightarrow R^{I}_{n}[w]\cap[\varphi]^{\mathfrak{M}^I}\neq\emptyset$.
    \end{itemize}
\end{lemma}
\begin{proof}
    We prove point $(i)$; point $(ii)$ will then follow by duality between $\square_\top$ and $\lozenge_\top$. We proceed by induction on $n$. 
    
    For $n=0$, we have that $\square_\top^0\varphi=\varphi$ and $R^I_0[w]=\{w\}$; therefore, clearly $\models^{\mathfrak{M}^I}_w \varphi \Leftrightarrow w \in [\varphi]^{\mathfrak{M}^I}$.

    For $n=m+1$, we have that $\square^{m+1}_\top \varphi=\square_\top\square^m_\top \varphi$. We then obtain the following:
    \begin{align*}
        \models^{\mathfrak{M}^I}_w \square_\top\square^m_\top \varphi
        &\text{ $\Leftrightarrow$ } \pi(w,[\top]^{\mathfrak{M}^I})\subseteq[\square_\top^m\varphi]^{\mathfrak{M}^I}\\
        &\text{ $\Leftrightarrow$ } R^I[w]\subseteq[\square_\top^m \varphi]^{\mathfrak{M}^I}& \text{ by definition of }R^I\\
        &\text{ $\Leftrightarrow$ for all } x \in R^I[w], \models^{\mathfrak{M}^I}_x \square_\top^m \varphi\\
        &\text{ $\Leftrightarrow$ for all } x \in R^I[w], R^I_m[x] \subseteq [\varphi]^{\mathfrak{M}^I} & \text{ by induction hypothesis}\\
        &\text{ $\Leftrightarrow$ }\bigcup\limits_{x \in R^I[w]}R^I_m[x] \subseteq [\varphi]^{\mathfrak{M}^I}\\
        &\text{ $\Leftrightarrow$ }R^I_{m+1}[w] \subseteq [\varphi]^{\mathfrak{M}^I} & \text{ by definition of }R^I_m
    \end{align*}
\end{proof}

Important properties of $I$-nest frames are that certain formulas can be evaluated at the root solely by inspecting which branches originate from it and that the global validity of certain formulas is accounted for by truth at the root. The following lemma establishes these useful properties.
 
\begin{lemma}
\label{lem:broommod}
Consider any proposition selection model $\mathfrak{M}^I=\langle \Omega^I, \pi^I, V \rangle$ based on an $I$-nest frame $\mathfrak{F}^I=\langle \Omega^I, \pi^I \rangle$. For all $n \in \mathbb{N}^+$, the following hold:
\begin{alignat*}{2}
\text{i.}\quad &
\models^{\mathfrak{F}^I} \square_\bot\lozenge^n_\top\square_\top\bot
\;\Leftrightarrow\;
\models^{\mathfrak{M}^I} \square_\bot\lozenge^n_\top\square_\top\bot
&\quad&
\models^{\mathfrak{F}^I} \neg\square_\bot\lozenge^n_\top\square_\top\bot
\;\Leftrightarrow\;
\models^{\mathfrak{M}^I} \neg\square_\bot\lozenge^n_\top\square_\top\bot
\\
\text{ii.}\quad &
\models^{\mathfrak{M}^I} \square_\bot\lozenge^n_\top\square_\top\bot
\;\Leftrightarrow\;
\models^{\mathfrak{M}^I}_0 \lozenge^n_\top\square_\top\bot
&\quad&
\models^{\mathfrak{M}^I} \neg\square_\bot\lozenge^n_\top\square_\top\bot
\;\Leftrightarrow\;
\models^{\mathfrak{M}^I}_0 \neg\lozenge^n_\top\square_\top\bot
\\
\text{iii.}\quad &
\models^{\mathfrak{M}^I}_0 \lozenge^n_\top\square_\top\bot
\;\Leftrightarrow\;
(n,n)\in\Omega^I
&\quad&
\models^{\mathfrak{M}^I}_0 \neg\lozenge^n_\top\square_\top\bot
\;\Leftrightarrow\;
(n,n)\notin\Omega^I
\end{alignat*}
\end{lemma}
\begin{proof}
Point (i) is straightforward since the interpretation of $\bot$ and $\top$ does not depend on the specific valuation function $V$. Point (ii) follows from Lemma \ref{lem:propofbroomframe}(ii). As for point (iii): clearly, $[\top]^{\mathfrak{M}^I} = \Omega^I$, hence $[\top]^{\mathfrak{M}^I} \neq \emptyset$. Moreover, by Lemma~\ref{Lemma:IteratedTruthConditions}, for all $w \in \Omega^I$, $\models^{\mathfrak{M}^I}_w \lozenge^n_\top\varphi \Leftrightarrow \exists x \in R^I_n[w]\text{ such that } \models^{\mathfrak{M}^I}_x \varphi$. Additionally, by Lemma~\ref{Lemma:IteratedTruthConditions} observe that for all $w \in \Omega^I$, $\models_w^\mathfrak{M} \square_\top \bot \Leftrightarrow R^I[w] =\emptyset$. By Lemma \ref{lem:propofbroomframe}(iii), we know that for $(i, j) \in \Omega^I$, $R^I[(i, j)]=\emptyset$ iff $i=j$ (note that if $j>i$, then $(i,j)\not\in\Omega^{I}$). Therefore, by Lemma \ref{lem:propofbroomframe} again, $\models^{\mathfrak{M}^I}_0 \lozenge^n_\top\square_\top \bot$ if and only if $\exists (i, j) \in R^I_n[0]\text{ such that } R^I[(i, j)]=\emptyset$; that is, if and only if $(i,j) = (n, n) \in R^I_n[0]$, which obtains just in case $(n, n) \in \Omega^I$.
\end{proof}

\begin{lemma}
    \label{lem:keyforuncountebalemany}
    For $I \subset \mathbb{N}^+$, consider the $I$-nest frame $\mathfrak{F}^I=\langle \Omega^I, \pi^I \rangle$, the following hold for all $n \in \mathbb{N}^+$:
\begin{flalign*}
     & (i) &&  \square_\bot\lozenge^n_\top\square_\top\bot \in \mathsf{L}(\mathfrak{F}^I) \Leftrightarrow n \notin I &\\
     & (ii) && \neg \square_\bot\lozenge^n_\top\square_\top\bot \in \mathsf{L}(\mathfrak{F}^I) \Leftrightarrow n \in I &  
     \end{flalign*}
\end{lemma}
\begin{proof}
   By Lemmas~\ref{lem:propofbroomframe} and \ref{lem:broommod}.
\end{proof}

From the preceding, we obtain our principal result.

    \begin{theorem}
    \label{Theorem:ContinuumCoatoms}
        The lattice of normal conditional logics $\NExt(\mathsf{CK})$ has exactly $2^{\aleph_0}$ many coatoms.
    \end{theorem}
    \begin{proof}
        Since $Fm$ is countable and every conditional logic is a subset of $Fm$, $\NExt(\mathsf{CK})$ has at most $2^{\aleph_0}$ coatoms.

        Let $S, T\subset \mathbb{N}^+$ be such that $S \neq T$. Where $\mathfrak{F}^S$ and $\mathfrak{F}^T$ are the $S$-nest and the $T$-nest frames (Definition~\ref{def:broom}), respectively, $\mathsf{L}(\mathfrak{F}^S)$ and $\mathsf{L}(\mathfrak{F}^T)$ are consistent normal conditional logics by Lemma \ref{Lemma:PropFrameInducesaNormalLogic}. Therefore, by Corollary~\ref{Corollary:Coatomicity}, there are coatoms $\mathsf{C}^S, \mathsf{C}^T$ in $\NExt(\mathsf{CK})$ such that $\mathsf{L}(\mathfrak{F}^S) \subseteq \mathsf{C}^S$ and $\mathsf{L}(\mathfrak{F}^T) \subseteq \mathsf{C}^T$. 
        
        Since $S \neq T$, assume without loss of generality that $n \in S$ but $n \notin T$ for some $n \in \mathbb{N}^+$. By Lemma~\ref{lem:keyforuncountebalemany}, $\neg\square_\bot\lozenge^n_\top\square_\top\bot \in \mathsf{C}^S$ and $\square_\bot\lozenge^n_\top\square_\top\bot \in \mathsf{C}^T$; therefore, $\mathsf{C}^S \neq \mathsf{C}^T$. This establishes that $\NExt(\mathsf{CK})$ must have at least $2^{|\mathbb{N}^+|}=2^{\aleph_0}$ coatoms, thereby completing the proof.
    \end{proof}

    The following results are corollaries of Theorem~\ref{Theorem:ContinuumCoatoms}.

 \begin{corollary}
        \label{Corollary:ContinuumCoatoms1}
        Each of the following lattices has exactly $2^{\aleph_0}$ many coatoms:
        \[
        \begin{gathered}
            \Ext(\mathsf{Cr}),\quad \Ext(\mathsf{Ck}),\quad
            \Ext(\mathsf{CR}),\quad \Ext(\mathsf{CK}),\\
            \rExt(\mathsf{Cr}),\quad \nExt(\mathsf{Ck}),\quad
            \RExt(\mathsf{CR}),\quad \NExt(\mathsf{CK}).
        \end{gathered}
        \]
    \end{corollary}
    \begin{proof}
        For any $I \subset \mathbb{N}^+$, since $\mathsf{L}(\mathfrak{F}^I)$ a normal conditional logic, it is also a semiregular, seminormal, and regular conditional logic; hence, $\mathsf{L}(\mathfrak{F}^I)$ occurs in each of the displayed lattices. Therefore, by Corollary~\ref{Corollary:Coatomicity} and adaption of the proof of Theorem~\ref{Theorem:ContinuumCoatoms}, each of these lattices has $2^{\aleph_0}$ many coatoms.
    \end{proof}

\begin{corollary}
    \label{corollary:CKextensions}
    $\mathsf{Cr}$, $\mathsf{Ck}$, $\mathsf{CR}$, and $\mathsf{CK}$ each have $2^{\aleph_0}$ many Post complete extensions.
\end{corollary}
\begin{proof}
    Immediate from Lemma~\ref{lem:post=coatoms} and Corollary~\ref{Corollary:ContinuumCoatoms1}.
\end{proof}

The foregoing results are consistent with the finite counts of the regular and normal Post complete conditional logics given in Theorem~\ref{Theorem:RegularPCLogics} and Corollary~\ref{Corollary:NormalPCLogics}. The important point to emphasize is that these Post complete systems do \textit{not} exhaust the coatoms of their respective restricted extension lattices. This allows us to strengthen Corollary~\ref{Corollary:EmbeddingFailure} as follows. 

\begin{corollary}
    \label{cor:embeddingfailur}
    There are $2^{\aleph_0}$ many logics in $\RExt(\mathsf{CR})$ ($\NExt(\mathsf{CK})$) not contained in any Post complete logic in $\RExt(\mathsf{CR})$ ($\NExt(\mathsf{CK})$).
\end{corollary}
\begin{proof}
By Corollary~\ref{Corollary:ContinuumCoatoms1}, $\RExt(\mathsf{CR})$ ($\NExt(\mathsf{CK})$) has $2^{\aleph_0}$ many coatoms, while by Theorem~\ref{Theorem:RegularPCLogics} (Corollary~\ref{Corollary:NormalPCLogics}) only nine (four) of them are Post complete. Each of the uncountably many remaining coatoms is a consistent regular (normal) conditional logic not contained in any Post complete regular (normal) conditional logic.
\end{proof}

Thus, there are uncountably many witnesses to the failure of the analogues of Makinson’s embedding theorem for regular and normal conditional logic.

These results also raise the question of characterizing the coatoms of the lattices considered here. In particular, we leave open whether the logic $\mathsf{CT} \oplus \text{\ref{Axiom:McK}}$ introduced above is a coatom of the lattice of normal conditional logics.

    We now turn to showing that \textsf{Ck} has $2^{\aleph_0}$ many seminormal Post complete extensions.
    
    %{\color{red} the trick is to look at a class of frames having the minimal property that we need, and not a single frame. The former construction wasn't working because the resulting logic of the frame would be inconsistent, if closed under uniform substituion. In fact, $p \boxright \bot$ was a theorem, and, by uniform susbtitution $X_k \boxright \bot$ would be a theorem too for some $X_k \boxright \bot \in \Gamma\setminus \Delta$. However, $\neg(X_k \boxright \bot)$ was a theorem too. Making the resulting logic inconsistent.}

    %{\color{blue} Good save, Giuliano. I think this works. I'll polish it up later today. I think we should maybe add some material explicitly defining substitutions, US, etc., to the section on logics.}

    \begin{definition}
        \label{Definition:CkRecursiveFormulae}
        For $i, k\geq 1$, we define:\footnote{So, for example, $X_2$ is $\top\wedge\top$ and $X_3$ is $\top\wedge\top\wedge\top$.}
            \begin{equation*}
                X_i = \bigwedge_{k\leq{i}}\top
            \end{equation*}
            \begin{equation*}
                %Y_i:=(\bigwedge_{k<{i}}\neg{(X_k \boxright\bot)})\wedge(X_i\boxright\bot)
                \Gamma=\{X_j \boxright \bot : j \in \mathbb{N}^{+}\}
            \end{equation*}
    \end{definition}

    \begin{definition}
        \label{Definition:CkRecursiveFrames}
            For $\Delta \subseteq \Gamma$, let $\mathcal{C}^\Delta$ be the class of all formula selection frames $\mathfrak{F}=\langle \{a\}, f\rangle$ satisfying the following two constraints:
            \begin{itemize}
                \item For all $k$ such that $X_k \boxright \bot \in \Gamma\setminus\Delta, f(a, X_k)=\{a\}$;
                \item For all $k$ such that $X_k \boxright \bot \in \Delta, f(a, X_k)=\emptyset$.
            \end{itemize}
    \end{definition}

    \begin{lemma}
        \label{Lemma:CDeltaLogics}
        For any $\Delta\subseteq\Gamma$, $\mathsf{L}(\mathcal{C}^{\Delta})$ is a consistent seminormal conditional logic.
    \end{lemma}
    \begin{proof}
        That $\mathsf{L}(\mathcal{C}^{\Delta})$ is a seminormal conditional logic is obvious \textit{except} for closure under uniform substitution, which we now verify. 
        
        We wish to show that, for any substitution $^\sigma$, if $\varphi\in\mathsf{L}(\mathcal{C}^{\Delta})$, then $\varphi^\sigma\in\mathsf{L}(\mathcal{C}^{\Delta})$. So suppose that $\varphi^\sigma\not\in\mathsf{L}(\mathcal{C}^{\Delta})$. Then for some model $\mathfrak{M}=\langle \mathfrak{F}, V \rangle$ over a frame $\mathfrak{F}=\langle \{a\},f\rangle\in\mathcal{C}^{\Delta}$, it must be that $\not\models^{\mathfrak{M}}_{a}\varphi^\sigma$. Consider the frame $\mathfrak{F}^\sigma=\langle \{a\},f^\sigma\rangle$ defined by putting:
        \[f^\sigma(a,\psi) = f(a,\psi^\sigma)\]
        Every formula $X_k$ is variable-free, so $f^\sigma(a,X_k) = f(a,{X_k}^\sigma) = f(a,X_k)$. Since $\mathfrak{F}\in\mathcal{C}^{\Delta}$, it follows from this and preceding that $\mathfrak{F}^\sigma\in\mathcal{C}^{\Delta}$ too. Now define the model $\mathfrak{M}^\sigma=\langle\mathfrak{F}^\sigma,V^\sigma\rangle$ by putting:
        \[V^\sigma(p) = [p^\sigma]^{\mathfrak{M}}\]
        We now wish to show by induction on the complexity of $\psi$ that:
        \[\models^{\mathfrak{M}}_{a}\psi^\sigma \text{ if and only if }\models^{\mathfrak{M}^\sigma}_{a}\psi\]
        We briefly examine the basis cases and the case concerning $\boxright$. Where $\psi$ is a variable $p$, $\models^{\mathfrak{M}}_{a}p^\sigma$ if and only if $a\in [p^\sigma]^{\mathfrak{M}}$ if and only if $a \in V^\sigma(p)$ if and only if $\models^{\mathfrak{M}^\sigma}_{a}p$. Where $\psi$ is $\bot$, both $\not\models^{\mathfrak{M}}_{a}\bot^\sigma$ (since $\bot^\sigma = \bot$) and $\not\models^{\mathfrak{M}^\sigma}_{a}\bot$, which suffices. Finally, $\models^{\mathfrak{M}}_{a}(\alpha\boxright\beta)^\sigma$ if and only if $\models^{\mathfrak{M}}_{a}\alpha^\sigma\boxright\beta^\sigma$ if and only if $f(a,\alpha^\sigma)\subseteq[\beta^\sigma]^\mathfrak{M}$ if and only if $f^\sigma(a,\alpha)\subseteq[\beta]^{\mathfrak{M}^\sigma}$ if and only if $\models^{\mathfrak{M}^\sigma}_{a}\alpha\boxright\beta$. The foregoing induction and the fact that $\not\models^{\mathfrak{M}}_{a}\varphi^\sigma$ imply that $\not\models^{\mathfrak{M}^\sigma}_{a}\varphi$, whence $\varphi\not\in\mathsf{L}(\mathcal{C}^{\Delta})$, which was to be proved.
        
        Finally, it is clear that for any $\Delta\subseteq\Gamma$, $\mathcal{C}^{\Delta}$ is nonempty, from which the consistency of $\mathsf{L}(\mathcal{C}^{\Delta})$ easily follows. 
    \end{proof}

    \begin{lemma}
        \label{Lemma:CkSeminormalExtensions}
        For any $\Delta\subseteq\Gamma$:
        \begin{itemize}
            \item[i.] $\mathsf{L}(\mathcal{C}^{\Delta})$ contains $\Delta$;
            \item[ii.] $\mathsf{L}(\mathcal{C}^{\Delta})$ contains $\{\neg(X_k \boxright \bot): X_k \boxright \bot\in\Gamma\setminus\Delta\}$;
            \item[iii.] $\mathsf{L}(\mathcal{C}^{\Delta})$ contains $(\varphi\leftrightarrow\psi)\to((\theta\boxright\varphi)\leftrightarrow(\theta\boxright\psi))$.
        \end{itemize} 
    \end{lemma}
    \begin{proof}

        Ad (i): If $X_k\boxright\bot\in\Delta$, then for any model $\mathfrak{M}$ over any frame $\mathfrak{F}\in\mathcal{C}^\Delta$, $f(a,X_k)=\emptyset\subseteq[\bot]^{\mathfrak{M}}$, and so $\models^{\mathfrak{M}}_{a}X_k \boxright \bot$. Therefore, $X_k\boxright\bot\in\mathsf{L}(\mathcal{C}^{\Delta})$.

        Ad (ii): If $X_k\boxright\bot\in\Gamma\setminus\Delta$, then for any model $\mathfrak{M}$ over any frame $\mathfrak{F}\in\mathcal{C}^\Delta$, $f(a,X_k)=\{a\}\not\subseteq[\bot]^{\mathfrak{M}}$, so $\models^{\mathfrak{M}}_{a}\neg(X_k \boxright \bot)$. Therefore, $\neg(X_k\boxright\bot)\in\mathsf{L}(\mathcal{C}^{\Delta})$.
        
        Ad (iii): Consider any model $\mathfrak{M}$ over any frame $\mathfrak{F}\in\mathcal{C}^\Delta$ such that $\models^{\mathfrak{M}}_{a}\varphi\leftrightarrow\psi$; then $[\varphi]^{\mathfrak{M}}=[\psi]^{\mathfrak{M}}$. Either $f(a, \theta)=\{a\}$ and $\{a\}\subseteq[\varphi]^{\mathfrak{M}}$ iff $\{a\}\subseteq[\psi]^{\mathfrak{M}}$, or $f(a, \theta)=\emptyset$ and $\emptyset\subseteq[\varphi]^{\mathfrak{M}}$ iff $\emptyset\subseteq[\psi]^{\mathfrak{M}}$. In either case, $\models^{\mathfrak{M}}_{a}(\theta\boxright\varphi)\leftrightarrow(\theta\boxright\psi)$, as desired.
    \end{proof}

    \begin{theorem}
        \label{Theorem:CkSeminormalExtensions}
        $\mathsf{Ck}$ has $2^{\aleph_0}$ many seminormal Post complete extensions.
    \end{theorem}
    \begin{proof}
        By Lemma~\ref{Lemma:CDeltaLogics}, for every $\Delta\subseteq\Gamma$, $\mathsf{L}(\mathcal{C}^\Delta)$ is a consistent seminormal extension of \textsf{Ck} having at least one Post complete extension (Lemma~\ref{Lemma:Lindenbaum}). Lemma~\ref{Lemma:CkSeminormalExtensions}(iii) guarantees that every extension of any system $\mathsf{L}(\mathcal{C}^\Delta)$ is closed under \ref{Rule:RCEC}; consequently, every Post complete extension of any $\mathsf{L}(\mathcal{C}^\Delta)$ must be seminormal. It remains to show that for $\Delta, \Delta'\subseteq\Gamma$, if $\Delta\neq\Delta'$, then $\mathsf{L}(\mathcal{C}^\Delta)$ and $\mathsf{L}(\mathcal{C}^{\Delta'})$ have no consistent mutual extensions. Suppose (wlog) that some $X_k\boxright\bot\in\Delta$ but $X_k\boxright\bot\not\in\Delta'$; then by Lemma~\ref{Lemma:CkSeminormalExtensions}(i), $X_k\boxright\bot\in\mathsf{L}(\mathcal{C}^\Delta)$, but by Lemma~\ref{Lemma:CkSeminormalExtensions}(ii), $\neg(X_k\boxright\bot)\in\mathsf{L}(\mathcal{C}^{\Delta'})$, which suffices. From the foregoing, \textsf{Ck} must have at least $2^{|\Gamma|}=2^{\aleph_0}$ many seminormal Post complete extensions, and by Lemma~\ref{Lemma:MaxPostCompleteExtensions}, it must have exactly that many seminormal Post complete extensions.
    \end{proof}

    \begin{corollary}
        \label{Corollary:CrSemiregularExtensions}
        $\mathsf{Cr}$ has $2^{\aleph_0}$ many semiregular Post complete extensions.
    \end{corollary}

    It is of some interest to note that the construction just deployed can be adapted (in fact, in a simplified form) to show that the basic normal $\omega$-multimodal logic $\mathsf{K}^\omega$ has $2^{\aleph_0}$ many normal Post complete extensions. The proof is given in Appendix~\ref{Section:Post completeness in multimodal logic}.

\section{Conclusions}
\label{Section:Conclusions}

To the best of our knowledge, the results presented here represent the first systematic investigation of the structure of lattices of conditional logics, specifically focusing on Post completeness. Although conditional logics are intuitively similar to (multi)modal logics and our treatment of the former has been informed by techniques used to treat the latter (see, especially, \cite{Tarski1934-35Extensions,Makinson1971,Segerberg1972PostCompleteness}), the results presented here also show that there are some significant differences (see, e.g., Theorem~\ref{Theorem:ContinuumCoatoms}). Hence, methods and results from modal logic cannot invariably or simply be carried over into the conditional setting.

From this study, at least two promising research trajectories emerge. The first involves a more algebraically oriented analysis of lattices of conditional logics. A natural starting point would be to characterize the structural properties of certain lattices by identifying splitting pairs as well as meet/join-irreducible elements. Such an inquiry would generalize existing results concerning the lattice of normal modal logics (cf.~Kracht~\cite[Ch. 7]{Kracht1999-KRATAT-3}) to the domain of conditional logics. Recent advances in abstract algebraic logic provide a robust framework for a dual approach in this direction. Specifically, the algebraizability results for variably strict conditional logics with respect to varieties of Boolean algebras with binary operators as developed by Rosella and Ugolini~\cite{RU2025} would facilitate a dual study of the lattice of variably strict conditional logics through their associated lattice of algebraic subvarieties. Our result that there are exactly two Post complete variably strict conditional logics (Corollary~\ref{Corollary:NormalVSLogics}) suggests that the lattice of subvarieties of variably strict conditional algebras introduced in \cite{RU2025} is atomic, possessing precisely two atoms. Since these algebraizability results can be extended to other families of normal conditional logics, the structural study of the lattice of normal conditional logics can proceed in parallel with a dual investigation of the lattice of corresponding algebraic subvarieties.

A second research direction would be to further investigate the  connections between lattices of conditional and (multi)modal logics. Identifying isomorphisms and embeddings is of both technical and conceptual interest. Technically, such connections would clarify the extent to which metatheoretic results concerning modal logics and their lattices can be transferred to the conditional domain, and vice versa. Our results naturally invite several questions. For instance, the parallels between the lattices of (semi)normal conditional logics and normal $\omega$-multimodal logics (see Appendix~\ref{Section:Post completeness in multimodal logic}), which each have uncountably many coatoms, prompt the question of whether these structures are isomorphic (we conjecture that they are not). 

Connected to this is the question of whether it is possible to identify and characterize the normal conditional logics whose lattices of normal extensions are structurally similar to the lattice of normal modal logics in salient ways---for example, in having only finitely many coatoms all of which are Post complete. We have seen that \textsf{M} satisfies such desiderata (Corollary~\ref{Corollary:ModalizedPCLogics} and Theorem~\ref{Theorem:CoatomsNExtP}) whereas \textsf{CK} does not (Corollary~\ref{Corollary:NormalPCLogics} and Theorem~\ref{Theorem:ContinuumCoatoms}). In Section~\ref{Subsection:Axiom Systems}, we remarked that in \textsf{M} the outer necessity ($\square\varphi:=\neg\varphi\boxright\bot$) is a normal modal operator. Now observe that the two Post complete modalized conditional logics correspond, via this operator, to the two Post complete normal modal logics: in $\mathsf{CR}\oplus\text{\ref{Axiom:Top}}$, $\square\varphi$ is equivalent to $\top$ (as in \textsf{VER}), while in $\mathsf{CR}\oplus\text{\ref{Axiom:Imp}}$, $\square\varphi$ is equivalent to $\varphi$ (as in \textsf{TR}). This suggests the conjecture that, for suitable families of normal conditional logics, the coatoms of the lattice of their normal extensions are determined by the Post complete extensions of the normal modal logic induced by their outer necessity. Addressing this question requires further study, which may also proceed in parallel with the algebraic investigation outlined above.

On a more philosophical note, investigating these connections might help us address the question of reducibility between necessity and conditionality: if one lattice demonstrates a higher degree of granularity while subsuming the other, it may suggest conceptual priority of the one notion over the other. This invites deeper reflection on whether conditionals are essentially necessities in disguise, or whether necessity is properly understood as a species of conditionality.

\appendix

\section{Post completeness in multimodal logic}
\label{Section:Post completeness in multimodal logic}

In this appendix, we examine in greater detail some of the connections between conditional and multimodal logic mentioned in Section \ref{Section:Post completeness}. We employ a construction analogous to the one from Definition~\ref{Definition:CkRecursiveFrames} to demonstrate that there are $2^{\aleph_0}$ many Post complete normal $\omega$-multimodal logics.\footnote{For some other results on Post completeness in multimodal logic, see \cite{MaChen2021,ChenMa2024}.} 

We begin by briefly presenting some features of normal $\omega$-multimodal logic. Consider a propositional language $\mathcal{L}^\omega$ containing a countable set of variables $Var = \{p, q, \dots\}$, the propositional constant $\bot$, and the primitive connectives $\neg$, $\vee$, and $\square_i$ (for every $i <\omega$). Let $Fm^\omega$ denote the set of formulas in $\mathcal{L}^\omega$.  

An \textit{$\omega$-multimodal logic} is a subset of $Fm^\omega$ that contains all classical tautologies and is closed under modus ponens and uniform substitution. Consider the following axioms and rule:
\begin{equation*}
   \label{Axiom:NEC}
   \tag{NEC}
   \square_i \top 
\end{equation*}
\begin{equation*}
   \label{Axiom:DIST}
   \tag{DIST}
   \square_i(\varphi \wedge \psi) \leftrightarrow (\square_i \varphi \wedge \square_i \psi) 
\end{equation*} 
\begin{equation*}
    \label{Rule:RE}
    \tag{RE}
        \frac{\vdash\varphi\leftrightarrow\psi}{\vdash\square_i \varphi \leftrightarrow \square_i \psi}
\end{equation*}
An $\omega$-multimodal logic is \textit{normal} if it contains \ref{Axiom:NEC}, \ref{Axiom:DIST}, and \ref{Rule:RE} (for every $i <\omega$). By $\textsf{K}^\omega$ we intend the smallest normal $\omega$-multimodal logic. $\omega$-multimodal extensions, Post complete extensions, etc., are defined as above (mutatis mutandis).

\begin{definition}
    \label{Definition:MultiKripkeFrame}
    A \textit{multimodal Kripke frame} is a structure $\mathfrak{F}^\omega=\langle{W, \{R_i\}_{i<\omega}}\rangle$ where $W\neq\emptyset$ is a set of worlds and each $R_i$ is a binary accessibility relation on $W$ (we define $R_i[w]=\{x \in W : wR_i x\}$).
\end{definition}    

\begin{definition}
    \label{Definition:MultiKripkeModel}
    A \textit{multimodal Kripke model} is a structure $\mathfrak{M}^\omega=\langle \mathfrak{F}^\omega,V\rangle$, where $\mathfrak{F}^\omega=\langle{W, \{R_i\}_{i<\omega}}\rangle$ is a multimodal Kripke frame and $V:Var\to\mathcal{P}(W)$ is a valuation.
\end{definition}

With respect to a multimodal Kripke model $\mathfrak{M}^\omega=\langle{W, \{R_i\}_{i<\omega}},V\rangle$ and world $w\in W$, the relation $\models^{\mathfrak{M}^\omega}_{w}$ is defined in the usual way for Boolean connectives with the following additional clauses for the modal operators ($[\varphi]^{\mathfrak{M}^\omega}=\{w \in W : \models^{\mathfrak{M}^\omega}_w \varphi\}$):
\begin{itemize}
    \item[$\square_i$.] $\models^{\mathfrak{M}^\omega}_{w}\square_i \varphi$ iff $R_i[w] \subseteq [\varphi]^{\mathfrak{M}^\omega}$.
\end{itemize}
We define that $\models^{\mathfrak{M}^\omega}\varphi$ iff $\models^{\mathfrak{M}^\omega}_{w}\varphi$ for all $w\in W$; and $\models^{\mathfrak{F}^\omega}\varphi$ iff $\models^{\mathfrak{M}^\omega}\varphi$ for every model $\mathfrak{M}^\omega$ over $\mathfrak{F}^\omega$. For a frame $\mathfrak{F}^\omega$, by $\mathsf{L}(\mathfrak{F}^\omega)$ is intended $\{\varphi\in Fm^\omega:\models^{\mathfrak{F}^\omega}\varphi\}$. 

We are now in a position to establish a result analogous to Theorem \ref{Theorem:CkSeminormalExtensions}. Let $\Gamma = \{ \square_i \bot : i < \omega \}$.
    
\begin{definition}
    \label{Definition:MuliRecursiveFrames}
    For $\Delta \subseteq \Gamma$, we define the multimodal Kripke frame $\mathfrak{F}^\Delta=\langle \{a\}, \{R_i\}_{i < \omega}\rangle$ by putting:
    \begin{equation*}
        R_i[a]=
        \begin{cases}
            \{a\} & \text{when } \square_i\bot\in \Gamma \setminus \Delta,
            \\\emptyset & \text{else}
        \end{cases}
    \end{equation*}
\end{definition}

The proofs of the following results are analogous to those culminating in Theorem~\ref{Theorem:CkSeminormalExtensions} above, and are omitted:

\begin{lemma}
    \label{Lemma}
    For any $\Delta\subseteq\Gamma$:
    \begin{itemize}
        \item[i.] $\mathsf{L}(\mathfrak{F}^{\Delta})$ is a consistent normal $\omega$-multimodal logic;
        \item[ii.] $\mathsf{L}(\mathfrak{F}^{\Delta})$ contains $\Delta$;
        \item[iii.] $\mathsf{L}(\mathfrak{F}^{\Delta})$ contains $\{\neg\square_k \bot: \square_k\bot \in\Gamma\setminus\Delta\}$;
        \item[iv.] $\mathsf{L}(\mathfrak{F}^{\Delta})$ contains $(\varphi\leftrightarrow\psi)\to(\square_i \varphi\leftrightarrow\square_i\psi)$ for every $i < \omega$.
    \end{itemize} 
\end{lemma}

\begin{theorem}
    \label{Theorem:MultiNormalExtensions}
    $\mathsf{K}^\omega$ has $2^{\aleph_0}$ many \textit{normal} Post complete extensions.
\end{theorem}

The foregoing reveal both similarities and differences between normal conditional and normal \(\omega\)-multimodal logic. Although their lattices have the same number of coatoms ($2^{\aleph_0}$), they differ in how many of those coatoms are Post complete (only four in the former, but $2^{\aleph_0}$ in the latter). A fuller comparison of the two lattice structures lies beyond the scope of this paper.

%The foregoing provides some insight into the structural relationship between multimodal logic and conditional logic. While normal conditional logics can be viewed as multimodal logics indexed by propositions, the lattice of normal conditional logics is \textit{not} isomorphic to the lattice of normal $\omega$-multimodal logics: the former contains exactly four coatoms, while the latter contains $2^{\aleph_0}$ coatoms---just like the lattice of seminormal conditional logics. Though a structural analysis of these two lattices, $\NExt(\mathsf{K}^\omega)$ and $\nExt(\mathsf{Ck})$, lies beyond the scope of this paper, this at least suggests that normal $\omega$-multimodal logics are more akin to seminormal conditional logics than normal conditional logics.

%\nocite{*} %to print all the biblio
%\bibliography{conditionals}
\printbibliography

\end{document}